\documentclass[11pt,letterpaper]{amsart}
\usepackage{amssymb, amscd, amsfonts, euler, epic, xy}
\usepackage{srcltx}
\xyoption{all}

\def\CC{{\mathbb C}}  
 
\def\FF{{\mathbb F}} 
\def\GG{{\mathbb G}} 
\def\HH{{\mathbb H}}

\def\PP{{\mathbb P}}
\def\QQ{{\mathbb Q}} 
\def\RR{{\mathbb R}}
\def\VV{{\mathbb V}} 
  
\def\ZZ{{\mathbb Z}}

\def\hor{{\rm hor}}

\def\ver{{\rm vert}}

\def\bs{\backslash}

\newcommand{\p}{\partial}

\def\Acal{{\mathcal A}}

\def\Ccal{{\mathcal C}}
\def\Dcal{{\mathcal D}}
\def\Ecal{{\mathcal E}} 
\def\Fcal{{\mathcal F}} 
\def\Hcal{{\mathcal H}} 
\def\Ical{{\mathcal I}}

\def\Lcal{{\mathcal L}}

\def\Ocal{{\mathcal O}}

\def\Mod{\text{M}}

\def\g{\gamma}
\def\la{\langle}
\def\ra{\rangle}

\def\cfrak{\mathfrak{c}}
\def\Dfrak{\mathfrak{D}}
\def\hfrak{\mathfrak{h}}
\def\gfrak{\mathfrak{g}}

\def\lfrak{\mathfrak{l}}
\def\lfrakhat{\hat{\mathfrak{l}}}
\def\nfrak{\mathfrak{n}}
\def\Lg{L {\mathfrak{g}}}
\def\Lghat{\widehat{L\mathfrak{g}}}

\def\Lcalhatg{\widehat{\Lcal\mathfrak{g}}}

\def\mfrak{\mathfrak{m}}

\def\Ubar{{\overline{U}}}

\def\End{\operatorname{End}}
\def\pt{{\scriptscriptstyle\bullet}}

\def\half{{\tfrac{1}{2}}}

\newcommand\ad{\operatorname{ad}}

\newcommand\aut{\operatorname{Aut}}

\newcommand\Gr{\operatorname{Gr}}

\newcommand\Hom{\operatorname{Hom}}

\newcommand\pic{\operatorname{Pic}}

\newcommand\res{\operatorname{Res}}
\newcommand\spec{\operatorname{Spec}}

\newcommand\supp{\operatorname{Supp}}
\newcommand\sym{\operatorname{Sym}}

\newcommand\Sp{\operatorname{Sp}}
\newcommand\Tr{\operatorname{Tr}}

\newtheorem{theorem}[equation]{Theorem}
\newtheorem{lemma}[equation]{Lemma}
\newtheorem{proposition}[equation]{Proposition}

\newtheorem{cordef}[equation]{Corollary-Definition}

\newtheorem{corollary}[equation]{Corollary}

\theoremstyle{definition}
\newtheorem{definition}[equation]{Definition}

\theoremstyle{remark} 
\newtheorem{remark}[equation]{Remark}

\title{From WZW models to Modular Functors}
\author{Eduard Looijenga}
\address{Mathematisch Instituut\\
Universiteit Utrecht\\
Postbus 80.010, NL-3508 TA Utrecht\\
Nederland}
\email{E.J.N.Looijenga@uu.nl}
\begin{document}

\begin{abstract}
In this survey paper we give a relatively simple and coordinate free description of the WZW model as a local system whose base is the $\GG_m$-bundle associated to the determinant bundle on the moduli stack of pointed curves. We derive its main properties and show how it leads to a modular functor in the spirit of  Segal.
The approach presented here is almost purely algebro-geometric in character; it avoids the
Boson-Fermion correspondence, operator product expansions as well as Teichm\"uller theory. 
\end{abstract}

\maketitle

The tumultuous interaction between mathematicians and theoretical physicists that began more than two decades ago left some of us hardly time to take stock. It is telling for this era that it took physicists (Witten, mainly) to point out in the late eighties that there must exist a bridge between two, at the time hardly connected, mathematical land masses, \emph{viz.}\ algebraic geometry and knot theory, and it is equally telling that it was only recently  that this was materialized with mathematically rigorous underpinnings (and strictly speaking not even in the desired form  yet).  We are here referring  on the algebro-geometric side to a subject that has its place in the present handbook, namely moduli spaces of vector bundles over curves, and on the other side to the knot invariants  (like  the Jones polynomial) that are furnished by Chern-Simons theory. The bridge metaphor is actually a bit misleading, because on either side the roads leading to it had yet to be constructed.  Let us use the remainder of this introduction to survey very briefly the part this route that involves algebraic geometry, stopping short at the  point  were the crossing is made,  then say which segment is covered by this paper and conclude in the customary manner by commenting on the various sections.
\\

To set the stage, let $C$ be a compact Riemann  surface and $G$ a 
(say, simply connected) complex algebraic group with simple Lie algebra $\gfrak$. Then there is a moduli stack
$\Mod (C,G)$ of $G$-principal bundles over $C$. This stack carries a natural ample line bundle $\Theta (C,G)$, which in fact generates its Picard group, and for which the vector space of sections of $\Theta (C,G)^{\otimes \ell}$, 
the so-called \emph{Verlinde space of level $\ell$} and here denoted by  $\HH_\ell(G)_C$,  is finite dimensional for all $\ell$. Its dimension is independent of $C$ and indeed,  if we vary $C$ over a base $S$,  then we get a vector bundle $\Hcal_\ell(G)_{\Ccal/S}$ over that base.  Although we required  $G$ to be simply connected, one can make sense of this for reductive groups as well, although some care is needed. For instance, for  $G=\CC^\times$,  we let $\Mod (C,\CC^\times)$ not be the full Picard variety $\pic (C)$ of $C$, but  pick the component $\pic (C)^{g-1}$ parameterizing line bundles of degree $(g-1)$, as this is the one which carries a natural line bundle that can  play  the role of $\Theta (C,\CC^\times)$ (and which is indeed known as the theta bundle). In that case $\HH_\ell(G)_C$ is just the  space of theta functions of degree $\ell$. These theta functions satisfy a heat equation and it is our understanding that Mumford was the first to observe that this property may be 
interpreted as  defining a flat connection for the associated projective space bundle. Hitchin \cite{hitchin} proved that this is 
also true for the case considered here: the  projectivized Verlinde bundles  come naturally with a flat connection. But if one aims for flat 
connections on the  bundles themselves, then one should work on the total space of a $\CC^\times $-bundle over $S$ 
(which allows for nontrivial monodromy in a fiber). For the line bundle attached to this $\CC^\times$-bundle we can take the 
determinant bundle  of the direct image of the sheaf of relative differentials on $\Ccal/S$.
For many purposes---certainly for topological applications---it is desirable to allow for certain `impurities' of the principal 
bundle, in the form of a parabolic structure. Such a structure is specified by giving on $C$ a finite set of points  $(x_i\in C)_{i\in I}$, and for each such point a finite dimensional  irreducible representation $V_i$ of $G$. It was shown by Scheinost-Schottenloher \cite{schsch} that in this  setting there are still corresponding Verlinde bundles  that come with a flat connection after a pull-back to a $\CC^\times$-bundle.
There is an infinitesimal counterpart of the above construction via holomorphic conformal field theory  where the group $G$ enters only via its Lie algebra $\gfrak$, known as the \emph{Wess-Zumino-Witten model}.  This centers on the affine Lie algebra associated to $\gfrak$ and its representation theory and leads to similar constructs such as the Verlinde bundles with a projectively flat connection. Its mathematically rigorous treatment  began with the fundamental paper by  Tsuchiya-Ueno-Yamada \cite{tuy} with subsequent extensions and refinements, mainly  by Andersen-Ueno \cite{jk1}, \cite{jk2}. It was however  not a priori  clear that this led to the same local system as the global approach. Indeed, this turned out to be not trivial at all: after partial results by Beauville-Laszlo and others, Laszlo-Sorger \cite{ls} proved that the Verlinde bundles can be identified  and 
Laszlo \cite{laszlo} showed that via this identification the two connections are the same as well, at least when no parabolic structure is present. 

The bridge is now crossed as follows: a nonzero point of the determinant line over $C$ can be topologically specified by means of the choice of Lagrangian sublattice in $H_1(C; \ZZ)$. This enables us to understand the existence of the flat connection on the Verlinde bundles as telling  us that  these spaces  only depend on the isotopy class of the  complex structure of $C$. In particular, they naturally receive  the structure of a projective representation of the mapping class group of the pointed surface. This puts these spaces into the topological realm and we thus arrive at an example of a topological quantum field theory, more precisely, at one of  Segal's modular functors \cite{segal}. 
\\

Let us now turn to the central goal of this paper, which is to define the  Wess-Zumino-Witten connection and to derive its principal properties, to wit its flatness, factorization, the relation with the KZ-system, $\dots$, in short, to recover  all the properties needed for defining the  underlying (topological) modular functor as found in  the papers above  mentioned by Tsuchiya-Ueno-Yamada and Andersen-Ueno. For an audience of algebraic geometers  knowing, or willing to accept, some rather basic facts about affine Lie algebras, our presentation is essentially self-contained. It  is also shorter and possibly at several points more transparent  than the literature we are aware of. This is to a large extent due to our consistent coordinate free approach, which not only has the advantage of making it unnecessary  to constantly check for gauge invariance, but is also conceptually more satisfying. Cases in point are our definition of the WZW-connection and our treatment of the Fock representation  (leading up to Corollary \ref{cor:twistedfock}) which enables us  to avoid resorting to the infinite wedge representation and allied  techniques. 

Let us take the occasion to point out that what makes the WZW-story still incomplete is an explanation of the duality property and the unitary structure that the associated modular functor should possess.
\\

As promised, we finish with brief comments on the contents of the separate sections. The rather short Section 1 essentially 
elaborates on the notion of a projectively flat connection. 
Logically, this material should have its place  later in the paper, but as it has some motivating content for what comes 
right after it, we felt it best to put it there.  Section 2 introduces in a canonical way the Virasoro algebra and its Fock 
representation and the associated Segal-Suguwara construction in a relative setting. New is the last subsection about 
symplectic local systems, where we see the determinant bundle appear in a canonical fashion. The Lie algebra $\gfrak$ 
enters in Section 3. We found it helpful to present this material in an abstract algebraic setting,  replacing for  instance the 
ring of complex Laurent polynomials by a complete local field containing $\QQ$ (or rather a direct sum of these), which is 
 then also allowed to `depend on parameters'. Our extension \ref{cor:derive} of the Sugawara representation to a relative 
situation involving a Leibniz rule in  the horizontal direction serves here as the origin of WZW-connection and its projective 
 flatness. We keep that setting  in Section 4, where the connection itself is defined.  In the 
subsequent section we derive the coherence of the Verlinde sheaf and  establish what is called the propagation of covacua. 
Special attention is paid to the genus zero case and it  shown how the WZW-connection is then related to the one of  
Knizhnik-Zamolodchikov. Section 6 is devoted to the basic results associated to a double point 
degeneration such as local freeness, factorization and monodromy. Finally, in Section 7, we establish the conversion into 
a modular functor. Notice that the approach described here is elementary and does not resort to Teichm\"uller theory. 

This paper is based on (but substantially supersedes) our arXiv preprint \texttt{math.AG/0507086}.

The author is grateful to two referees for helpful comments.
\\

We find it convenient to work over an arbitrary algebraically closed field $k$ of characteristic zero, but in Section 7, where we discuss the link with topological quantum field theory,  we assume $k=\CC$.  As an intermediate base we use a regular $k$-algebra, denoted  $R$.

\section{Flat and  projectively flat connections}\label{section:projflat}
A central notion of this article is that of a  flat  projective connection. 
Although it enters the scene much later in the paper, some of the work done in the first part is motivated 
by the particular way this notion appears here. So we  start with a brief section discussing it.

We begin with  recalling some basic facts. Let $\Hcal$ be a rank $r$ vector bundle over a smooth base $S$ (in other words, is a locally a free  $\Ocal_S$-module of  rank $r$).  Then the  Lie algebra $\Dcal_1(\Hcal)$ of first order differential operators 
$\Hcal\to\Hcal$  fits in an exact sequence of coherent sheaves of Lie algebras
\[
0\to \Ecal nd(\Hcal)\to\Dcal_1(\Hcal){\buildrel sb\over\longrightarrow} \theta_S\otimes_{\Ocal_S} \Ecal nd(\Hcal)\to 0,
\]
where $sb$ is the symbol map which assigns to $D\in\Dcal_1(\Hcal)$ the $k$-derivation 
$\phi\in\Ocal_S\mapsto [D,\phi] \in \Ocal_S$.  The local sections of  $\Dcal_1(\Hcal)$ whose symbol land in $\theta_S\otimes 1_\Hcal\cong \theta_S$ make up a coherent  subsheaf of Lie subalgebras $sb^{-1}(\theta_S)\subset \Dcal_1(\Hcal)$ so that we have
an exact sequence of coherent sheaves of Lie algebras
\[
0\to \Ecal nd(\Hcal)\to sb^{-1}(\theta_S)\to \theta_S\to 0.
\]
A connection on $\Hcal$ is then simply a section $X\in\theta_S\mapsto \nabla_X\in sb^{-1}(\theta_S)$ of   $sb^{-1}(\theta_S)\to \theta_S$ and it is flat precisely if this section is a Lie homomorphism.
This suggests that we define a flat connection on the associated projective space bundle $\PP_S(\Hcal)$ as a Lie subalgebra $\hat\Dcal\subset \Dcal_1(\Hcal)$ with $sb(\hat\Dcal)=\theta_S$ and $\hat\Dcal\cap \Ecal nd(\Hcal)=\Ocal_S\otimes 1_\Hcal\cong \Ocal_S$ so that
we have an exact subsequence
\[
0\to \Ocal_S\to\hat\Dcal\to \theta_S\to 0
\]
 of the one displayed above. That this is a sensible definition follows from the observation that an $\Ocal_S$-linear section $\nabla$ of $\hat\Dcal\to \theta_S$ 
defines a connection  on $\Hcal$ whose curvature form $R(\nabla)$ is a closed  $2$-form on $S$.   Any other section $\nabla'$ differs from
$\nabla$ by an $\Ocal_S$-linear map $\theta_S\to \Ocal_S$, in other words, by a differential $\omega$, 
and we have  $R(\nabla')=R(\nabla)+d\omega$. So this indeed gives rise to a flat connection in  
$\PP_S(\Hcal)$ and it is easily seen that this connection is independent of the choice of the section. 
Locally on $S$,  $R(\nabla)$ is exact, and so we can always choose $\nabla$  to be flat as a connection. Any other such local section that is flat is necessarily of the form 
$\nabla +d\phi$ with $\phi\in \Ocal$ and conversely, any such local section has that property. The Lie algebra sheaf  
$\hat\Dcal$ itself does not determine a connection on $\Hcal$; this is most evident  when $\Hcal$ is a line bundle, 
for then we must have $\hat\Dcal(\Hcal)=\Dcal_1(\Hcal)$.

In the above situation  we let $\hat\Dcal$ act on the  determinant bundle $\det (\Hcal)=\wedge^r_{\Ocal_S}\Hcal$ by means of the formula
\[
\hat D(e_1\wedge\cdots\wedge e_r):=\sum_{i=1}^r e_1\wedge\cdots\wedge \hat D(e_i)\wedge\cdots\wedge e_r.
\]
This is indeed well-defined, and identifies $\hat\Dcal$ as a Lie algebra with the Lie algebra of first order differential operators 
$\Dcal_1(\det (\Hcal))$. Notice however that this identification makes  $f\in \Ocal_S\subset \hat\Dcal$ act on $\det (\Hcal)$  as multiplication by $rf$.  

Let us next observe that if $\lambda$ is a line bundle on $S$ and $N$ is a positive integer, then a similar formula identifies 
$\Dcal_1(\lambda)$ with $\Dcal_1(\lambda^{\otimes N})$ (both as $\Ocal_S$-modules and as $k$-Lie algebras), but  
induces multiplication by $N$ on $\Ocal_S$.  This leads us to make the following

\begin{definition}\label{def:lambdaflat} 
Let be given a smooth base variety $S$ over which we are given a line bundle $\lambda$ and  a locally free 
$\Ocal_S$-module $\Hcal$ of finite rank.  A \emph{$\lambda$-flat connection}  on $\Hcal$ is  homomorphism  of 
$\Ocal_S$-modules  $u:\Dcal_1(\lambda)\to  \Dcal_1(\Hcal)$ that is also a Lie homomorphism over $k$,  commutes 
with the symbol maps (so these must  land in $\theta_S$) and takes scalars to scalars: $\Ocal_S\subset \Dcal_1(\lambda)$ is mapped to
$\Ocal_S\subset\Dcal_1(\Hcal)$.
\end{definition}

It follows from the preceding that such a homomorphism $u$ determines
a flat connection on the projectivization of $\Hcal$. The map $u$ preserves  $\Ocal_S$ and 
since this restriction is $\Ocal_S$-linear, it is given by multiplication by some regular function $w$ on $S$.
If  $D\in\theta_S$ is lifted to $\hat D\in\Dcal_1(\lambda)$, then $u(\hat D)\in  \Dcal_1(\Hcal)$ is also a lift of $D$ and so
$D(w)= [u(\hat D), u(1)]=u([\hat D,1])=0$. 
This shows that $w$ must be locally constant;  we call this  the \emph{weight} of $u$. So in the above discussion, 
$\hat\Dcal$ comes with  $\det (\Hcal)$-flat connection of weight $r^{-1}$.

It is clear that if the weight of $u$ is constant zero, then $u$ factors through $\theta_S$, so that we get a flat 
connection in $\Hcal$. This is also the case when $\lambda=\Ocal_S$, for then $\Dcal_1(\Ocal_S)$ contains 
$\theta_S$  canonically  as a direct summand (both as $\Ocal_S$-module and as a sheaf of $k$-Lie algebras) 
and the flat connection is then given by the action of $\theta_S$. This has an interesting consequence: if
$\pi: \Lambda^\times\to S$ is the geometric realization of the  $\GG_m$-bundle defined by 
$\lambda$, then  $\pi^*\lambda$
has a `tautological' generating section and thus gets identified with $\Ocal_{\Lambda^\times}$. Hence
a  $\lambda$-flat connection on $\Hcal$ defines an ordinary  flat connection on $\pi^*\Hcal$. One checks that if $w$ is the 
weight of $u$, then the connection is homogeneous of degree $w$ along the fibers. So in case  $k=\CC$,  $s\in S$ and 
$\tilde s\in \Lambda^\times$ lies over $s\in S$, then  the multivalued map 
$(z, h)\in\CC^\times\times H_s \mapsto (z\tilde s, z^wh)\in \Lambda^\times_s\times H_s$  is flat,  and so the 
monodromy of the connection in $\Lambda^\times_s$  is scalar multiplication by $e^{2\pi\sqrt{-1}w}$.

We will also encounter a logarithmic version. Here we are given a closed subvariety $\Delta\subset S$ of lower dimension
(usually a normal crossing hypersurface). Then  the $\theta_S$-stabilizer of the ideal defining $\Delta$, denoted  
$\theta_S(\log\Delta)$, is  a coherent $\Ocal_S$-submodule of $\theta_S$ closed under the Lie bracket.  If
in Definition \ref{def:lambdaflat} we have $u$  only  defined on the preimage of $\theta_S(\log\Delta)\subset \theta_S$ in 
$\Dcal_1(\lambda)$  (which we denote here by $\Dcal_1(\lambda)(\log \Delta)$), then we say that we have  a  
\emph{logarithmic $\lambda$-flat connection relative to $\Delta$}  on $\Hcal$.

\section{The Virasoro algebra and its basic representation}\label{section:vir} 

Much of the material exposed in this section is a conversion of certain standard constructions
(as can be found for instance in \cite{kacraina}) 
into a coordinate invariant and relative setting. But the way we introduce the Virasoro algebra is less standard and may be even new. A similar remark applies to part of the last subsection (in particular, Corollary \ref{cor:twistedfock}), which is devoted to the Fock module attached to a symplectic local system.
\\

In this section we fix an $R$-algebra $\Ocal$ isomorphic to the formal power series ring $R[[t]]$.
In other words, $\Ocal$  comes with a principal ideal $\mfrak$
so that $\Ocal$ is complete for the $\mfrak$-adic topology and the associated graded $R$-algebra
$\oplus_{j=0}^\infty \mfrak^j/\mfrak^{j+1}$ is a polynomial ring over $R$ in one variable.
The choice of a generator $t$ of the ideal $\mfrak$ identifies $\Ocal$ with $R[[t]]$. We denote by
$L$ the localization of $\Ocal$ obtained by inverting a generator of $\mfrak$.
For $N\in\ZZ$, $\mfrak^N$ has the obvious meaning
as an $\Ocal$-submodule of $L$. The \emph{$\mfrak$-adic topology} on $L$ is the topology that has the 
collection of cosets $\{ f+\mfrak^N\}_{f\in L,N\in
\ZZ}$ as a basis of open subsets. We sometimes write $F^NL$ for $\mfrak^{N}$. We further denote by
$\theta$ the $L$-module of continuous $R$-derivations from $L$ into $L$
and by $\omega$ the $L$-dual of $\theta$. These $L$-modules 
come with filtrations (making them principal filtered $L$-modules): $F^N\theta$ consists of the derivations that take 
$\mfrak$ to 
$\mfrak^{N+1}$ and $F^N\omega$ consists of the $L$-homomorphisms $\theta\to L$ that take  
$F^0\theta$ to $\mfrak^N$.
So in terms of the generator $t$ above, $L=R((t))$, $\theta=R((t))\frac{d}{dt}$, 
$F^N\theta=R[[t]]t^{N+1}\frac{d}{dt}$,  $\omega =R((t))dt$ and $F^N\omega=R[[t]]t^{N-1}dt$. 

The residue map
$\res: \omega\to R$ which assigns to an element of $R((t))dt$ the coefficient of $t^{-1}dt$
is canonical, i.e., is independent of the choice of $t$. The $R$-bilinear map
\[
r: L\times \omega \to R,\quad  (f,\alpha)\mapsto \res (f\alpha)
\]
is a topologically perfect pairing of filtered $R$-modules: we have $r(t^k, t^{-l-1}dt)=\delta_{k,l}$ and 
so any $R$-linear  $\phi: L\to R$  which is continuous (i.e., $\phi$ zero on $\mfrak^N$ for some $N$) 
is definable by an element of $\omega$
(namely by $\sum_{k>N} \phi (t^{-k})t^{k-1}dt$) and likewise
for an $R$-linear continuous map $\omega\to R$.

\subsection*{A trivial Lie algebra}
If we think of of the multiplicative group $L^\times$ of $L$ as an algebraic group over $R$ (or rather, as a group object
in a category of ind schemes over $R$), then
its Lie algebra, denoted here by $\lfrak$, is $L$,  regarded as a $R$-module with trivial Lie bracket. 
It comes with a decreasing filtration $F^\pt\lfrak$ (as a Lie algebra) defined by the valuation. 
The universal enveloping algebra $U\lfrak$ is clearly the symmetric algebra of $\lfrak$ as an $R$-module, 
$\sym_R^\pt (\lfrak)$. 
The ideal $U_+\lfrak\subset U\lfrak$ generated by $\lfrak$
is also a right  $\Ocal$-module (since $\lfrak$ is). We complete it $\mfrak$-adically:
given an integer $N\ge 0$, then an $R$-basis of the truncation
$U_+\lfrak/(U\lfrak\circ F^N\lfrak)$ 
is the collection $t^{k_1}\circ\cdots\circ t^{k_r}$ with 
$k_1\le k_2\le \cdots \le k_r<N$. So elements of the  completion 
\[
U_+\lfrak\to \Ubar_+\lfrak:=\varprojlim_N U_+\lfrak/U\lfrak\circ F^N\lfrak
\]
are series of the form 
$\sum_{i=1}^\infty r_i t^{k_{i,1}}\circ\cdots\circ t^{k_{i,r_i}}$
with $r_i\in R$, $c\le k_{1,i}\le k_{2,i}\le\cdots \le k_{i,r_i}$ for some constant $c$. 
We put $\Ubar\lfrak:= R\oplus \Ubar_+\lfrak$, which we could  of course 
have  defined just as well directly as
 \[
U\lfrak\to \Ubar\lfrak:=\varprojlim_N U\lfrak/U\lfrak\circ F^N\lfrak.
\]
We will refer to this construction as the \emph{$\mfrak$-adic completion on the right}, although
in the present case there is no difference with the analogously defined $\mfrak$-adic completion 
on the left,  as  $\lfrak$ is commutative.

Any  continuous derivation $D\in\theta$ defines an $R$-linear map $\omega\to L$ which is 
self-adjoint relative the residue pairing: $r(\la D,\alpha \ra,\beta)=r(\alpha,\la D,\beta\ra)$.
We use that pairing to identify $D$ with an
element of the closure of $\sym^2\lfrak$ in $\Ubar\lfrak$. Let $C(D)$ be half this element, 
so that in terms of the above topological basis,
\[
C(D)=\half \sum_{i,j\in\ZZ} r(\la D, t^{-i-1}dt\ra, t^{-j-1}dt) t^i\circ t^j.
\]
In particular for $D=D_k=t^{k+1}\frac{d}{dt}$, $C(D_k)=\half \sum_{i+j=k} t^i\circ t^j$. Observe 
that the map  $C: \theta\to\Ubar\lfrak$ is continuous. 

\subsection*{Oscillator and Virasoro algebra}
The residue map defines a central extension of $\lfrak$, the 
\emph{oscillator algebra} $\lfrakhat$, which as an $R$-module is 
simply $\lfrak\oplus R$. If we denote the generator of the second summand by $\hbar$, then the Lie bracket is given by 
\[
[f +\hbar r ,g+\hbar s ]:=\res (g\, df)\hbar.
\]
So $[t^{k},t^{-l}]=k\delta_{k,l}\hbar $ and the center of  $\lfrakhat$ is 
$Re\oplus R\hbar$, where $e=t^0$ denotes the unit element of $L$ viewed as an element of $\lfrak$.  
It follows that $U\lfrakhat$ is an $R[e, \hbar ]$-algebra.  As an $R[\hbar]$-algebra it is obtained as follows: 
take  the tensor algebra of 
$\lfrak$ (over $R$) tensored with $R[\hbar]$, $\otimes^\pt_R \lfrak \otimes_R R[\hbar]$, and divide that out 
by the two-sided ideal generated by the elements $f\otimes g-g\otimes f-
\res (gdf)\hbar$.  The obvious surjection 
$\pi: U\lfrakhat\to U\lfrak=\sym_R^\pt (\lfrak)$ is the reduction modulo $\hbar$.
 
We filter $\lfrakhat$ by letting $F^N\lfrakhat$ be $F^N\lfrak$ for $N>0$
and $F^N\lfrak +R\hbar$ for $N\le 0$. This filtration is used to complete 
$U\lfrakhat$ $\mfrak$-adically on the right:
\[
U\lfrakhat\to \Ubar\lfrakhat:=\varprojlim_N U\lfrakhat /U\lfrakhat\circ F^N\lfrak.
\]
Notice that this completion has the collection 
$t^{k_1}\circ\cdots\circ t^{k_r}$ with $r\ge 0$, $k_1\le k_2\le \cdots \le k_r$,
as topological $R[\hbar]$-basis.  Since $\hat\lfrak$ is not abelian, the left and  right
$\mfrak$-adic topologies now  differ. For instance, 
$\sum_{k\ge 1} t^k\circ t^{-k}$ does not converge in $\Ubar\lfrakhat$, whereas 
$\sum_{k\ge 1} t^{-k}\circ t^k$ does. The obvious surjection 
$\pi: \Ubar\lfrakhat\to \Ubar\lfrak$ is still given by reduction modulo $\hbar$.
We also observe that the filtrations of $\lfrak$ and $\hat\lfrak$ determine decreasing 
filtrations of  their (completed) universal enveloping algebras, e.g.,
$F^NU\hat\lfrak=\sum_{r\ge 0}\sum_{n_1+\cdots +n_r\ge N} 
F^{n_1}\hat\lfrak\circ\cdots \circ F^{n_r}\hat\lfrak$.

Let us denote by $\lfrak_2$ the image of 
$\lfrak\otimes_R\lfrak\subset\lfrakhat\otimes_R\lfrakhat\to U\lfrakhat$.
Under the reduction modulo $\hbar$,  $\lfrak_2$ maps onto $\sym^2_R(\lfrak)\subset U\lfrak$ with kernel $R\hbar$.
Its  closure $\bar \lfrak_2$ in $\Ubar\lfrakhat$ maps onto the closure of 
$\sym^2_R(\lfrak)$ in $\Ubar \lfrak$ with the same kernel. 

The generator $t$ defines a  continuous $R$-linear map $D\in\theta\mapsto  \hat C(D)\in \bar \lfrak_2$ 
characterized by 
\[
\hat C(D_k):= \half\sum_{i+j=k} : t^i\circ t^j:\; .
\]
We here adhered to the \emph{normal ordering convention}, which prescribes that the factor with the 
highest index comes last and hence acts first (here the exponent serves as index). So $: t^i\circ t^j:$ equals
$t^i\circ t^j$ if $i\le j$ and $t^j\circ t^i$ if $i> j$.
This map is clearly a  lift of $C: \theta\to \sym^2\lfrak$, but is otherwise non-canonical.

\begin{lemma}\label{lemma:chat}
We have
\begin{enumerate}
\item[(i)] $[\hat C(D),f]= -\hbar D(f)$ as an identity  in $\Ubar \lfrakhat$ 
(where $f\in\lfrak\subset\hat\lfrak$) and
\item[(ii)] $[\hat C(D_k), \hat C(D_l)]=-\hbar(l-k)  
\hat C(D_{k+l})+\hbar^2\frac{1}{12}(k^3-k)\delta_{k+l,0}$.
\end{enumerate}
\end{lemma}
\begin{proof}
For the first statement we compute $[\hat C( D_k), t^l]$. If we substitute
$\hat C(D_k)=\half\sum_{i+j=k} : t^i\circ t^j:$, then we see that only terms of the
form $[t^{k+l}\circ t^{-l}, t^l]$  or $[t^{-l}\circ t^{k+l}, t^l]$ (depending on whether $k+2l\le 0$ or $k+2l\ge 0$) 
can make a contribution
and then have coefficient 
$\half$ if $k+2l= 0$ and $1$ otherwise. In all cases the result is 
$-\hbar lt^{k+l}= -\hbar D_k(t^l)$.
 
Formula (i) implies that 
\begin{multline*}
[\hat C(D_k), \hat C(D_l)]=
\lim_{N\to\infty}\sum_{|i|\le N} \half\left( D_k(t^i)\circ t^{l-i}+ t^i\circ D_k(t^{l-i})\right)\\
=-\hbar \lim_{N\to\infty}\sum_{|i|\le N} 
\left( it^{k+i}\circ t^{l-i}+ t^i\circ (l-i)t^{k+l-i}\right).
\end{multline*}
This is up to a reordering equal to $ -\hbar (l-k)\hat C(D_{k+l})$.
The terms which do not commute and are in the wrong order are those
for which $0<k+i=-(l-i)$ (with coefficient $i$) and for which
$0<i=-(k+l-i)$ (with coefficient $(l-i)$). This accounts for
the extra term $\hbar^2\frac{1}{12}(k^3-k)\delta_{k+l,0}$.
\end{proof}

This lemma shows that $-\hbar^{-1}\hat C$ behaves better than $\hat C$ (but requires us of course  
to assume that $\hbar$ be invertible). In fact, it suggests to consider  the set $\hat\theta$  of  pairs 
$(D, u)\in\theta\times \hbar^{-1}\bar\lfrak_2$  for which $C(D)\in\sym^2\lfrak$ is the mod $\hbar$ 
reduction of $-\hbar u$, so that we have an exact sequence
\[
0\to R \to \hat\theta \to \theta \to 0
\]
of $R$-modules.  Then a non-canonical section of $\hat\theta \to \theta$ is given by $D\mapsto \hat D:= (D, -\hbar^{-1}\hat C(D))$.
In order to avoid confusion, we denote the generator of the copy of $R$ by $c_0$.

\begin{cordef}\label{def:virasoro} 
This defines a central extension of Lie algebras, called the \emph{Virasoro algebra} (of the $R$-algebra $L$). 
Precisely,  if $T:  \hat\theta\to \Ubar\lfrakhat [\tfrac{1}{\hbar}]$ is given by the second component, then  
$T$ is injective, maps $\hat\theta$ onto a Lie subalgebra of 
$\Ubar\lfrakhat[\frac{1}{\hbar}]$  and sends $c_0$ to $1$. If we transfer the Lie bracket on $\Ubar\lfrakhat[\frac{1}{\hbar}]$ to $\hat\theta$, 
then in terms of  our non-canonical section, 
\[
[\hat D_k, \hat D_l]= (l-k)\hat D_{k+l}+\frac{k^3-k}{12}\delta_{k+l,0}c_0. 
\]
Moreover, $\ad_{T(\hat D)}$ leaves $\lfrak$ 
invariant (as a subspace of $\Ubar\hat\lfrak$) and acts on that subspace by derivation with respect to $D\in\theta$.
\end{cordef} 

\begin{remark}
An alternative coordinate free definition of the Virasoro algebra, based on the algebra of
pseudo-differential operators on $L$, can be found in \cite{bms}.
\end{remark} 

\subsection*{Fock representation}
It is clear that $F^0\lfrakhat=R\hbar\oplus \Ocal$ is an abelian
subalgebra of $\lfrakhat$. We let $F^0\lfrakhat =\Ocal \oplus R\hbar$ act on a free rank one module $Rv_o$ by 
letting $\Ocal$ act trivially and $\hbar$ as the identity. The induced representation of $\lfrakhat$ over $R$,
\[
\FF:= U\lfrakhat \otimes_{U\lfrakhat\circ F^0\lfrakhat}Rv_o,
\]
will be regarded as a $U\lfrak[\hbar^{-1}]$-module.
It comes with an increasing PBW (Poincar\'e-Birkhoff-Witt) filtration  $W_\pt\FF$ by $R$-submodules,  
with $W_r\FF$ being the image of  $\oplus_{s\le r}\lfrakhat^{\otimes s}\otimes Rv_o$. 
Since the scalars $R\subset\lfrak$ are central in $\hat\lfrak$ and kill $\FF$ (because 
$R\subset \Ocal$), they act trivially in all of $\FF$. 
As an $R$-module, $\FF$ is free with basis  the collection $t^{-k_r}\circ\cdots\circ t^{-k_1}\otimes v_o$, 
where  $r\ge 0$ and $1\le k_1\le k_2\le \cdots\le k_r$ (for $r=0$, read $v_o$). (In fact,  
$\Gr^W_\pt\FF$ can be identified as a graded $R$-module with the symmetric algebra 
$\sym^\pt (\lfrak/F^0\lfrak)$.) This also shows that 
$\FF$ is even a $\Ubar\lfrakhat[\hbar^{-1}]$-module. Thus $\FF$ affords a representation 
of $\hat\theta$ over $R$, called its 
\emph{Fock representation}.  

It follows from Lemma \ref{lemma:chat} that for any $D\in\theta$ with lift $\hat D\in\hat\theta$, 
\begin{multline*}
T(\hat D)t^{-k_r}\circ\cdots \circ t^{-k_1}\otimes v_o=\\
=\left(\sum_{i=1}^r t^{-k_r}\circ\cdots\circ
D(t^{-k_i})\circ\cdots \circ t^{-k_1}\right) \otimes v_o+t^{-k_r}\circ\cdots \circ t^{-k_1}\circ T(\hat D) v_o.
\end{multline*}
Since $T(\hat D)v_o=0$ when $D\in F^0\theta$, it follows that
$F^0\theta$ acts on $\FF$
by coefficient-wise derivation.
This observation has an interesting consequence. Consider the module of  $k$-derivations $R\to R$ (denoted here simply 
by $\theta_{R}$ instead of the more accurate $\theta_{R/k}$)
and the module $\theta_{L,R}$ of $k$-derivations of $L$ that are continuous for the
$\mfrak$-adic topology and  preserve $R\subset L$.
Since $L\cong R((t))$ as an $R$-algebra, every $k$-derivation 
$R\to R$ extends to  one from $L$ to $L$. So we have an exact sequence
\[
0\to \theta \to \theta_{L,R}\to \theta_R\to 0.
\]
The following corollary essentially says that we have defined in the $L$-module $\FF$ a Lie algebra 
$\hat\theta_{L,R}$ of first order ($k$-linear) differential operators  which contains $R$ as the degree zero 
operators and for which the  symbol map (which is just the formation of the degree one quotient) has image 
$\theta_{L,R}$.

\begin{corollary}\label{cor:derivevir}
The actions on $\FF$ of $F^0\theta_{L,R}=\mfrak\theta_{\Ocal,R}\subset\theta_{L,R}$  (given by 
coefficient-wise derivation,  killing the generator $v_o$) and  $\hat\theta$ coincide on $F^0\theta$ and 
generate a central extension of Lie algebras $\hat\theta_{L,R}\to \theta_{L,R}$ by $Rc_o$. Its defining representation on 
$\FF$ (still denoted $T$)  is faithful and has the property that for every lift $\hat D\in\hat \theta_{L,R}$ of 
$D\in\theta_{L,R}$ and $f\in \lfrak$ we have $[T(\hat D),f]=Df$ (in particular, it preserves every $U\hat\lfrak$-submodule 
of $\FF$).
\end{corollary}
\begin{proof}
The generator $t$ can be used to define a section of $\theta_{L,R}\to \theta_R$:
the set of elements of $\theta_{L,R}$ which kill $t$ is a $k$-Lie subalgebra
of  $\theta_{L,R}$ which projects isomorphically onto $\theta_R$.
Now if $D\in \theta_{L,R}$, write $D=D_\ver+D_\hor$ with $D_\ver\in \theta$ and 
$D_\hor (t)=0$ and define an $R$-linear operator
$\hat D$ in $\FF$ as the sum of $T(\hat D_\ver)$ 
and  coefficient-wise derivation by $D_\hor$. This map clearly has the
properties mentioned. 

As to its  dependence on $t$: another choice yields a decomposition of the form $D=(D_\hor+D_0) +(D_\ver -D_0)$  with 
$D_0\in F^0\theta$ and in view of the above $\hat D_0$ acts in $\FF$ by coefficient-wise derivation.
\end{proof}

\subsection*{The Fock representation for a symplectic local system}

In Section \ref{section:wzw} we shall run into  a particular type of finite rank subquotient of the Fock representation 
and it seems best to discuss the resulting structure  here. We  start out from the following data:
\begin{enumerate}
\item[(i)] a free $R$-module $H$ of finite rank endowed with a symplectic form $\la\; ,\;\ra :H\otimes_RH\to R$, which is 
nondegenerate in the sense that the induced map $H\to H^*$, $a\mapsto \la\; ,a\ra$ is an isomorphism of $R$-modules, 
\item[(ii)] an $R$-submodule  $\Dfrak\subset\theta_R$ closed under the Lie bracket for which the inclusion is  an equality 
over the generic point and a Lie action $D\mapsto \nabla_D$ of  $\Dfrak$ on $H$ by $k$-derivations which preserves the 
symplectic form,
\item[(iii)] a Lagrangian $R$-submodule $F\subset H$.
\end{enumerate}
Property (ii) means that  $D\in\Dfrak\mapsto \nabla_D\in \End_k(H)$ is $R$-linear, obeys the Leibniz rule: 
$\nabla_D(ra)=r\nabla_D(a)+D(r)a$ and satisfies $\la \nabla_Da,b\ra +\la a, \nabla_Db\ra = D\la a,b\ra$. 
In the cases of interest,  $\Dfrak$  will be the $\theta_R$-stabilizer of a principal ideal in $R$ (and often be all 
of $\theta_{R}$).  One might think of $\nabla$ as a flat meromorphic connection on the symplectic  bundle represented by $H$.

In this setting, a Heisenberg algebra  is defined in an obvious manner:  it is $\hat H:=H\oplus R\hbar $  endowed 
with the bracket $[a+R\hbar, b+R\hbar]=\la a,b\ra\hbar$.  We also have defined a Fock representation $\FF (H,F)$ of 
$\hat H$ as the induced module of the rank one representation of  $\hat F=F +R\hbar$ on $R$ given by the coefficient 
of $\hbar$. Notice that if we grade $\FF (H,F)$ with respect to the PBW filtration, we get a copy of the symmetric 
algebra of $H/F$ over $R$. We aim to define a projective Lie action of $\Dfrak$ on $\FF (H,F)$.

We begin with extending the $\Dfrak$-action to $\hat H$ by stipulating that it kills $\hbar$. This action clearly 
preserves the Lie bracket and hence determines one of $\Dfrak$ on the universal enveloping algebra $U\hat H$. 
This does not however induce one in $\FF (H,F)$, as $\nabla_D$ will not respect the right ideal in $U\hat H$ 
generated by $\hbar-1$ and  $F$.  We will remedy this by means of a `twist'. 

We shall use the isomorphism  $\sigma :H\otimes_RH\cong \End_R (H)$ of $R$-modules defined by  associating to 
$a\otimes b$ the endomorphism $\sigma (a\otimes b): x\in H\mapsto a\la b,x\ra\in H$. If we agree to identify an 
element in the tensor algebra of $H$, in particular, an element of $H$, as the operator in $U\hat H$ or $\FF(H,F)$ 
given  by left multiplication, then it is ready checked that for $x\in H$, 
\[
[a\circ b, x]=\sigma (a\otimes b+b\otimes a)(x).
\]
We choose a Lagrangian supplement of $F$ in $H$, i.e.,  a Lagrangian $R$-submodule $F'\subset H$ that is also a 
section of $H\to H/F$. Since $F'$ is an abelian Lie subalgebra  of 
$\hat H$, we have a natural map  $\sym^\pt_R (F')\to \FF (H,F)$. It is clearly an isomorphism of 
$\sym^\pt_R (F')$-modules.
 Now write $\nabla_D$ according to the Lagrangian decomposition $H=F'\oplus F$:
\[
\nabla_D=\begin{pmatrix} \nabla^{F'}_D & \sigma'_{D}\\
\sigma_D & \nabla^F_D
\end{pmatrix}.
\]
Here the diagonal entries represent the induced connections on $F'$ and $F$, whereas $\sigma_D\in \Hom_R(F',F)$ and 
$\sigma'_{D}\in \Hom_R(F,F')$. Since $\sigma$ identifies $F\otimes_R F$ resp.\  $F'\otimes_R F'$ with 
$\Hom_R(F',F)$ resp.\ $\Hom_R(F,F')$, we can write $\sigma_D=\sigma (s_D)$ with $s_D\in F\otimes_RF$ and 
$\sigma'_{D}=\sigma(s'_D)$ and $s'_D\in F'\otimes_R F'$. These tensors are symmetric and represent the 
second fundamental form of $F'\subset H$ resp.\ $F\subset H$. 
Notice that if $a\in F$, then 
\[
[\nabla_D, a]=\nabla_D(a)=\nabla^F_D(a) +\sigma^F_{F'}(a)=\nabla^F_D(a)+\half [s'_D,a]
\]
and  similarly,  if $a'\in F'$, then $[\nabla_D, a']=\nabla^{F'}_D(a')+\half [s_D,a']$. This suggests we should assign to
$D\in\Dfrak$ the first order differential operator $T_{F'}(D)$ in $\FF(H,F)\cong\sym^\pt F'$ defined by
\[
T_{F'}(D):= \nabla^{F'}_D+\half s_D+\half s'_D. 
\]

\begin{proposition}\label{propdef:connection}
The map $T_{F'}:\Dfrak\to \End_k(\sym^\pt F')$ is $R$-linear and has the property that $[T_{F'}(D),a]=\nabla_D(a)$ 
for every $D\in\Dfrak$ and $a\in \hat H$. Any other map $\Dfrak\to \End_k(\sym^\pt F')$ enjoying these properties  
differs from $T_{F'}$ by a multiple of the identity operator, in other words, is of the form
$D\mapsto T_{F'}(D)+\eta (D)$ for some $\eta\in\Hom_R(\Dfrak, R)$. 
\end{proposition}
\begin{proof}
That $T_{F'}(D)$ has the stated property follows from the preceding. Let $\eta:\Dfrak\to \End_k(\sym^\pt F')$ 
be the difference of two such maps. Then for every $D\in\Dfrak$,  $\eta(D)\in\End_R(\FF(H,F))$ commutes with all 
elements of $\hat H$. Since $\FF(H,F)$ is irreducible as a representation of $\hat H$, it follows that $\eta(D)$ is a scalar 
in $R$.
\end{proof}\index{\footnote{}}

Notice that if $u_1,\dots ,u_r\in \hat H$, then
\begin{multline*}
T_{F'}(D)(u_r\circ\cdots\circ u_1\otimes v_o)=\\=
\left(\sum_{i=1}^r u_r\circ\cdots \circ \nabla_D(u_i)\circ\cdots\circ u_1+ u_r\circ\cdots\circ u_1\circ \half s'_D\right)\otimes v_o.
\end{multline*}
So this looks  like the operator $T_{\hat D}$ acting in $\FF$  with $s'_D$ playing the role of $-\hat C(D)$. 
Here is the key result about the `curvature' of $T_{F'}$.

\begin{lemma}
Given $D,E\in\Dfrak$, then  $[T_{F'}(D),T_{F'}(E)]-T_{F'}([D,E])$ is scalar multiplication by $1/2$ times the value on 
 of the $\nabla^F$-curvature on $\det (F)$ on the pair $(D,E)$. 
\end{lemma}
\begin{proof}
The fact that $\nabla$ preserves the Lie bracket is expressed by the following identities:
\begin{align*}
\nabla^{F}_D\nabla^{F}_E-\nabla^{F}_E\nabla^{F}_D-\nabla^{F}_{[D,E]}&=\sigma_E\sigma'_D-\sigma_D\sigma'_E,\\
\nabla^{F'}_D\nabla^{F'}_E-\nabla^{F'}_E\nabla^{F'}_D-\nabla^{F'}_{[D,E]}&=\sigma'_E\sigma_D-\sigma'_D\sigma_E,\\
\nabla^{\Hom(F',F)}_D(\sigma_E)-\nabla^{\Hom(F',F)}_E(\sigma_D)&=\sigma_{[D,E]},\\
\nabla^{\Hom(F,F')}_D(\sigma'_E)-\nabla^{\Hom(F,F')}_E(\sigma'_D)&=\sigma'_{[D,E]}.
\end{align*}
The first two give the curvature of $\nabla^F$ and $\nabla^{F'}$ on the pair $(D,E)$. 
The last two can also be written  as operator identities in $\sym^\pt F'$:
\begin{align*}
[\nabla^{F'}_D,s_E]-[\nabla^{F'}_E,s_D]&=s_{[D,E]},\\
[\nabla^{F'}_D,s'_E]-[\nabla^{F'}_E,s'_D]&=s'_{[D,E]}.
\end{align*}
If we feed these identities in:
\begin{multline*}
[T_{F'}(D),T_{F'}(E)]-T_{F'}([D,E])=\\=
 [\nabla^{F'}_D+\half s_D+\half s'_D, \nabla^{F'}_E+\half s_E+\half s'_E]- (\nabla^{F'}_{[D,E]}+\half s_{[D,E]}+
 \half s'_{[D,E]})=\\
=\big([\nabla^{F'}_D,\nabla^{F'}_E]-\nabla^{F'}_{[D,E]}\big)
+\half \big([\nabla^{F'}_D,s_E]-[\nabla^{F'}_E,s_D]-s_{[D,E]})\big)\\
+\half \big([\nabla^{F'}_D,s'_E]-[\nabla^{F'}_E,s'_D]-s'_{[D,E]})\big)
+\tfrac{1}{4}\big([s_D,s'_E]-[s_{E},s'_D]\big)
\end{multline*}
(where we identified $\FF (H,F)$ with $\sym^\pt F'$), we obtain
\[
[T_{F'}(D),T_{F'}(E)]-T_{F'}([D,E])=
\big(\sigma'_E\sigma_D+\tfrac{1}{4}[s_D,s_E'])-\big(\sigma'_D\sigma_E +\tfrac{1}{4}[s_{D'},s_E]\big).
\]
We must show that the right hand side is equal to $\half\Tr (\sigma_E\sigma'_D-\sigma_D\sigma'_E)$,
or perhaps more specifically, that  $\sigma'_E\sigma_D+\tfrac{1}{4}[s_D,s_E']=- \half\Tr (\sigma_D\sigma'_E)$ (and 
similarly if we exchange $D$ and $E$).
This reduces to the following identity in linear algebra: if $a\in F$ and $\beta\in F'$, then in $\sym^\pt F'$ we have
\[
\sigma(\beta\otimes\beta)\sigma (a\otimes a)+\tfrac{1}{4}[a\circ a,\beta\circ\beta]=-\half
\Tr_{F'}(\sigma_{a\otimes a}\sigma_{\beta\otimes\beta}),
\]
Indeed, a straightforward computation shows that
\[
[a\circ a,\beta\circ\beta]=2\la a,\beta\ra (a\circ\beta+\beta\circ a)=4\la a,\beta\ra\beta\circ a+2\la a,\beta\ra^2.
\]
If we interpret $\la a,\beta\ra\beta\circ a$ as an operator in $\sym^\pt F'$, then applying it to  
 $x\in F'$ yields $\la a,\beta\ra\beta\la a, x\ra=-\sigma(\beta\otimes\beta)\sigma (a\otimes a)(x)$. We also find that 
$\la a,\beta\ra^2=-\Tr_{F'}(\sigma(a\otimes a)\sigma(\beta\otimes\beta)$.
\end{proof}

If $N$ is a free $R$-module of rank one, then by a \emph{square root of $N$} we mean a free $R$-module 
$\Theta$ of rank one together with an  isomorphism of $\Theta\otimes_R\Theta$ onto $N$.

\begin{corollary}\label{cor:twistedfock}
Let  $\Theta$ be a square root of $\det_R(F)$.  Then the twisted Fock module 
$\Hom_R (\Theta, \FF(H,F))$ comes with a natural action of $\Dfrak$ by derivations.
\end{corollary}
\begin{proof}
Given the Lagrangian supplement $F'$ of $F$ in $H$, then endow $\Theta$ with the unique 
$\Dfrak$-module structure that makes the given isomorphism $\Theta\otimes_R\Theta\cong \det_R(F)$ 
one of  $\Dfrak$-modules: if $w\in \Theta$ is a generator and $\nabla^{\det F}_D(w\otimes w)=rw\otimes w$, 
then $\nabla^\Theta_D(w)=\half rw$.
This ensures that the $\Dfrak$-action on  $\Hom_R(\Theta, \sym^\pt F')$ preserves the Lie bracket. It remains to 
show that this action is independent of $F'$. This can be verified by a computation, but rather than carrying this out, 
we give an abstract argument that avoids this. It is based on the well-known fact that if $H_o$ is a 
fixed symplectic  $k$-vector space of finite dimension $2g$, and  $F_o\subset H_o$ is  Lagrangian, then the set 
of  Lagrangian supplements of $F_o$ in $H_o$ form in the Grassmannian of $H_o$ an affine space over 
$\sym^2_kF_o$ (and hence is simply connected). Now by doing the preceding construction universally over 
the corresponding affine space over 
$\sym^2_R F$, we see that the flatness on the universal example immediately gives the independence. 
\end{proof}

\begin{remark}\label{rem:twistedfock}
We will use this corollary mainly via the following  reformulation. First we observe that
the Lie algebra of first order $k$-linear differential operators $\Theta\to \Theta$ projects to $\theta_{R}$ 
(this is the symbol map) with kernel the scalars $R$. Denote by $\Dfrak(\Theta)$ the preimage of $\Dfrak$. 
This is clearly a Lie subalgebra. Then the above corollary can be understood as saying that there is a natural 
Lie action of $\Dfrak(\Theta)$ on $\FF (H,F)$ by first order differential operators, acting, in the terminology of 
Section \ref{section:projflat}, with weight $1$. The image in $\End_k(\FF (H,F))$ is the $R$-submodule of 
$\End_k(\FF (H,F))$ generated by the $T_{F'}(D)$ and the identity operator.  
We may also use $\Dfrak(\det_R(F))$ instead, although then the weight will be $\half$.
Note that our discussion of projectively flat connections in Section 1 now suggests a formulation in more 
geometric terms, namely that the pull-back of $\FF (H,F)$ to the geometric realization of the $\GG_m$-bundle 
over $\spec (R)$ defined by  $\det_R(F)$ acquires a flat meromorphic connection with fiber monodromy 
minus the identity. 
\end{remark}

\begin{remark}
 The preceding follows the presentation of Boer-Looijenga \cite{bl} rather closely. The quadratic terms that enter in the definition of $T_{F'}$ are in a way a relict of the heat operator of which the theta functions associated to this symplectic local system are solutions (flat sections are expansions of theta functions relative to an unspecified lattice).
\end{remark}

\section{The Sugawara construction}\label{section:sugawara}

In this section we show how the Virasoro algebra acts in the standard representions
of a centrally extended loop algebra. This construction goes back to the physicist H.~Sugawara (in 1968), but it was probably Graeme Segal who first noticed its relevance for the present context.  

Most of the material below can for instance be found in \cite{kacraina} (Lecture 10) and   \cite{kac} (Ch.\ 12), but our presentation slightly deviates from the standard sources in substance as well in form: we approach the Sugawara construction via the construction discussed in the previous section and put it in the (coordinate free) setting that makes it appropriate for the application we have in mind. 
\\

In this section, we fix a simple Lie algebra $\gfrak$ over $k$ of finite dimension. 
We retain the data and the notation of Section \ref{section:vir}.

\subsection*{Loop algebras} We identify  $\gfrak\otimes\gfrak$ with the space of  bilinear forms $\gfrak^*\times\gfrak^*\to k$, where  $\gfrak^*$ denotes  the $k$-dual of $\gfrak$, as usual. We form its space of $\gfrak$-covariants (relative to the 
adjoint action on both factors):
\[
q:\gfrak\otimes\gfrak\to (\gfrak\otimes\gfrak)_\gfrak=:\cfrak.
\]
This space is known to be of dimension one and to consist of symmetric tensors. 
It has a canonical generator which is characterized by the 
property that it is represented by a $\gfrak$-invariant symmetric tensor $c\in \gfrak\otimes\gfrak$ with the property that
$c(\theta\otimes\theta)=2$ if $\theta\in\gfrak^*$ is a long root (relative to a choice of Cartan subalgebra $\hfrak$ of 
$\gfrak$; the roots then lie in the zero eigenspace of $\hfrak$ in $\gfrak^*$). This element is in fact  invariant under 
the full  automorphism group of the Lie algebra $\gfrak$, not just the inner ones. It is nondegenerate when viewed 
as a symmetric bilinear form  on $\gfrak^*$  and so the inverse form on  $\gfrak$ is defined. If we denote the latter by 
$\check{c}$,  then the equivariant projection $q:\gfrak\otimes\gfrak\to \cfrak$ is given by 
$X\otimes Y\mapsto \check{c}(X,Y)c$.

It is well-known and easy to prove that $c$ maps to the center of $U\gfrak$. This implies that
$c$ acts in any irreducible representation  of $\gfrak$ by a scalar.
In the case of the adjoint representation half this scalar is called the \emph{dual Coxeter number} of $\gfrak$  
and is denoted by $\check{h}$. So if we choose an orthonormal basis $\{ X_\kappa\}_\kappa$ 
of $\gfrak$ relative to $\check{c}$ so that $c$  takes the form $\sum_\kappa X_\kappa\otimes X_\kappa$, then 
\[
\sum_\kappa [X_\kappa ,[X_\kappa, Y]]=2\check{h}Y\quad \text{for all $Y\in\gfrak$.}
\]

Let $L\gfrak$ stand for $\gfrak\otimes_k L$, but considered as a filtered $R$-Lie algebra (so we restrict the scalars to 
$R$) with $F^N L\gfrak= \gfrak\otimes_k \mfrak^N$.
An argument similar to the one we used to prove that the pairing $r$ is  topologically  perfect shows that the pairing
\[
r_\gfrak : (\gfrak\otimes_k L)\times (\gfrak\otimes_k\omega)\to \cfrak\otimes_k R=:\cfrak_R
\]
which sends $(Xf,Y\alpha)$ to $q(X\otimes Y)\res (f\alpha)$ is topologically perfect
(the basis dual to  $(X_\kappa t^l)_{\kappa,l}$ is $(X_\kappa t^{-l-1}dt\otimes c)_{\kappa,l}$). 

For an integer $N\ge 0$, the
quotient $U L\gfrak/U L\gfrak\circ F^NL\gfrak$ is a free $R$-module  (a set of generators is 
$X_{\kappa_1}t^{k_1}\circ\cdots\circ X_{\kappa_r}t^{k_r}$, $k_1\le\cdots\le k_r<N$). 
We complete $U L\gfrak$  $\mfrak$-adically on the right: 
\[
\Ubar L\gfrak:= \varprojlim_N  U L\gfrak/
U L\gfrak\circ  F^NL\gfrak.
\]
A central extension of Lie algebras
\[
0\to\cfrak_R\to \Lghat\to \Lg\to 0
\]
is defined by endowing the sum $L\gfrak\oplus \cfrak_R$ with the Lie bracket
\[
[Xf+c r ,Yg+c s ]:=[X,Y]fg +r_\gfrak (Yg, Xdf).
\]
Since the residue is zero on $\Ocal$, the inclusion of 
$\Ocal\gfrak$ in $\Lghat$ is a homomorphism of Lie algebras. In fact, this is a  canonical (and even unique)
Lie section of the central extension over $\Ocal\gfrak$, for it is just the derived Lie algebra of the preimage
of $\Ocal\gfrak$ in $\Lghat$. The $\aut (\gfrak)$-invariance of $c$ implies
that the tautological action of $\aut (\gfrak)$ on $\gfrak$ extends to $\Lghat$.

We filter $\Lghat$ by setting  
$F^N\Lghat=F^N\Lg$ for $N>0$ and  $F^N\Lghat=F^N\Lg +\cfrak_R$ for $N\le 0$.
Then $U \Lghat$ is a filtered $R[c]$-algebra whose reduction modulo $c$ is 
$U L \gfrak$. The $\mfrak$-adic completion on the right
\[
\Ubar \Lghat:= \varprojlim_N  
U \Lghat/(U \Lghat\circ F^N\Lg)
\]
is still an $R[c]$-algebra and the obvious surjection 
$\Ubar\Lghat\to \Ubar L\gfrak$ is the reduction modulo $c$.
These (completed) enveloping algebras not only come with the (increasing) Poincar\'e-Birkhoff-Witt filtration, 
but also inherit a (decreasing)  filtration from $L$. 

\subsection*{Segal-Sugawara representation}
Tensoring with $c\in \gfrak\otimes_k\gfrak$ defines the $R$-linear map
\[
\lfrak\otimes_R\lfrak\to \Lg\otimes_R \Lg,\quad
f\otimes g\mapsto c\cdot f\otimes g=\sum_{\kappa}X_\kappa f\otimes X_{\kappa}g,
\]
which, when composed with $\Lg\otimes_R \Lg\subset \Lghat\otimes_R \Lghat\to U\Lghat$, yields a map 
$\g: \lfrak\otimes_R\lfrak\to U\Lghat$. Since 
$\g (f\otimes g-g\otimes f)=\sum_\kappa [X_\kappa f,X_\kappa g]=c\dim\gfrak \res (g df)$,  $\g$  
drops and extends naturally to an $R$-module homomorphism  
$\hat\gamma: \lfrak_2\to U\Lghat $ which sends $\hbar$ to  $c\dim\gfrak$.
This, in turn, extends continuously to a map from the closure $\bar\lfrak_2$ of
$\lfrak_2$ in $\Ubar \hat\lfrak$ to $\Ubar\Lghat$.
As $\bar\lfrak_2$ contains the image of $\hat C:\theta\to \Ubar\hat\lfrak$, and since $c$ is 
$\aut (\gfrak)$-invariant, we get a $R$-homomorphism 
\[
\hat C_\gfrak:=\hat\gamma\hat C :\theta\to (\Ubar \Lghat)^{\aut (\gfrak)}.
\]
We may  also describe $\hat C_\gfrak$ in the spirit of Section 
\ref{section:vir}: given $D\in \theta$, then the $R$-linear  map
\[
1\otimes D:  \gfrak\otimes_k\omega \to  \gfrak\otimes_k L
\]
is continuous and self-adjoint relative to $r_\gfrak$ and the perfect pairing 
$r_\gfrak$ allows us to identify it with an element of $\Ubar L\gfrak$; this element
produces  our $\hat C_\gfrak(D)$. Thus the choice of the parameter $t$ yields
\[
\hat C_\gfrak(D_k)=\half\sum_{\kappa, l}: X_\kappa t^{k-l}\circ 
X_\kappa t^l:\quad.
\]
This formula can be used to define  $\hat C_\gfrak$, but this approach does  not exhibit its naturality.

\begin{lemma}
For $X\in\gfrak$ and $f\in L$ we have
\[
[\hat C_\gfrak(D_k), Xf]= -(c+\check{h} )XD_k(f)
\]
(an identity in $\Ubar \Lghat$) and upon a choice of a parameter $t$, then
with the preceding notation
\[
[\hat C_\gfrak(D_k), \hat C_\gfrak(D_l)]=(c+\check{h})
(k-l)\hat C_\gfrak(D_{k+l})+c(c+\check{h})\delta_{k+l,0} \frac{k^3-k}{12}\dim \gfrak.
\]
\end{lemma}

For the proof (which  is  a bit  tricky, but not very deep), we refer to Lecture 10 of 
\cite{kacraina} (our  $C_\gfrak(\hat D_k)$ is their $T_k$). This formula suggests that we make the
central element  $c+\check{h}$ of $\Ubar \Lghat$ invertible (its inverse might be viewed  as a 
rational function on $\cfrak^*$), so that we can state this lemma  in a more natural manner as follows. 

\begin{corollary}[Sugawara representation]\label{cor:sugawara}
The map $\hat D_k\mapsto \frac{-1}{c+\check{h}}\hat C_\gfrak(D_k)$ induces a natural 
homomorphism of $R$-Lie algebras
\[
T_\gfrak: \hat\theta\to \big(\Ubar\Lghat [ \tfrac{1}{c+\check{h}}]\big)^{\aut(\gfrak)}
\]
which  sends the central element $c_0\in \hat\theta$ to $c(c+\check{h})^{-1}\dim \gfrak$. Moreover, if 
$\hat D\in\hat\theta$, then 
$\ad_{T_\gfrak(\hat D)}$ leaves $\Lg$ 
invariant (as a subspace of $\Ubar\Lghat$) and acts on that subspace by derivation with respect to the image of 
$\hat D$ in $\theta$.
\end{corollary}

\subsection*{A representation for $\Lghat$} 
We fix $\ell\in k$ with $\ell\not= -\check{h}$. Let $F^1\Lg\oplus \cfrak_R$ act on the free $R$-module of rank one 
$Rv_\ell$ via the projection onto the second factor $\cfrak_R=Rc$ with
$c$ acting as multiplication by $\ell$. We regard  $F^1\Lg\oplus \cfrak_R$ as a subalgebra of $U\Lghat$ so that we can 
form the induced module 
\[
\FF_\ell(\gfrak,L):=U\Lghat \otimes_{U(F^1\Lg\oplus \cfrak_R)} Rv_\ell,
\]
which we often simply denote by  $\FF_\ell(\gfrak)$.
We  use $v_\ell$ also to denote its image in this module. As  an $R$-module $\FF_\ell(\gfrak)$ is generated by
$X_{\kappa_r}t^{-k_r}\circ\cdots \circ X_{\kappa_1}t^{-k_1}\otimes v_\ell$, where $r\ge 0$, 
$0\le k_1\le k_2\le \cdots \le k_r$ and
where $(X_{\kappa})_\kappa$ is a given $k$-basis of  $\gfrak$.
If we let $\hat\theta$ act on  $\FF_\ell(\gfrak)$ via $T_\gfrak$, then 
it follows from Corollary \ref{cor:sugawara} that if $\hat D\in \hat\theta$ lifts $D\in\theta$, then
\begin{multline*}
T_\gfrak(\hat D)X_{\kappa_r}t^{-k_r}\circ\cdots \circ X_{\kappa_1}t^{-k_1}\otimes v_\ell=\\
=\sum_{i=1}^r X_{\kappa_r}t^{-k_r}\circ\cdots
X_{\kappa_i}D(t^{-k_i})\circ\cdots \circ X_{\kappa_1}t^{-k_1}\otimes v_\ell+\\
+X_{\kappa_r}t^{-k_r}\circ\cdots \circ X_{\kappa_1}t^{-k_1}\circ T_\gfrak(\hat D) v_\ell.
\end{multline*}
Thus $\hat\theta$ is  faithfully represented as a Lie algebra 
of $R$-linear endomorphisms of $\FF_\ell(\gfrak)$.
If $D\in F^0\theta$, then clearly $T_\gfrak(\hat D)v_\ell=0$ and hence
we have the following counterpart of Corollary \ref{cor:derivevir} (with
the same proof). It tells us that $\hat\theta_{L,R}$ acts in  $\FF_\ell(\gfrak)$ as  a Lie algebra 
of first order differential operators, but with its degree zero part $R$ acting with weight $(c+\check{h})^{-1}c\dim \gfrak$:

\begin{corollary}\label{cor:derive}
The Sugawara representation $T_\gfrak$ of $\hat\theta$ on  $\FF_\ell(\gfrak)$ 
extends to $\hat\theta_{L,R}$ in such  a manner that $F^0\theta_{L,R}$ acts by coefficientwise derivation 
(killing the generator $v_\ell$),  $[T_{\gfrak}(\hat D),Xf]=X(Df)$ for $X\in\gfrak$, $f\in L$ and $T_{\gfrak}(\hat D)$ is 
$\aut (\gfrak)$-invariant. In particular, this action preserves every  $U\Lghat$-submodule of $\FF_\ell(\gfrak)$.
\end{corollary}

\subsection*{Semi-local case}
This refers to the situation where we allow the $R$-algebra $L$ to be a finite direct sum of $R$-algebras isomorphic 
to $R((t))$: $L=\oplus_{i\in I} L_i$, where $I$ is a nonempty finite index set and $L_i$ as before. 
We then extend  the notation employed earlier in the most natural fashion.
For instance, $\Ocal$, $\mfrak$, $\omega$, $\lfrak$ are now the direct sums over $I$ (as filtered objects)  
of the items suggested by the notation.
If $r: L\times\omega\to R$ denotes the sum of the residue pairings of 
the summands, then $r$ is still topologically perfect. However, we take for 
the oscillator algebra $\hat\lfrak$ not  the direct sum of the $\hat\lfrak_i$, but rather
the quotient of  $\oplus_i\hat\lfrak_i$ that identifies the central generators 
of the summands with a single $\hbar$. We thus get a Virasoro extension
$\hat\theta$ of $\theta$ by $c_0 R$ and a (faithful) oscillator representation of
$\hat\theta$ in $\Ubar\hat\lfrak$. The decreasing filtrations are the obvious
ones. We shall denote by $\FF$ the Fock representation $\FF$ of $\lfrakhat$ 
that ensures that the unit of every summand $\Ocal_i$ acts  the identity; it is then the induced
representation of the rank one representation of $F^0\lfrakhat=\Ocal\oplus R\hbar$ in 
$Rv_o$.

In likewise manner we define $\Lghat$ (a central extension of 
$\oplus_{i\in I} L\gfrak_i$ by $\cfrak_R$) and construct  the associated Sugawara representation.
The representation $\FF_\ell(\gfrak)$ of  $\Lghat$ is as before. We have defined  $\hat\theta_{L,R}$
and Corollaries \ref{cor:sugawara} and  \ref{cor:derive} continue to hold.

\section{The WZW connection: algebraic aspects}\label{section:wzw}

From now on we place ourselves in the semi-local case, so $L=\oplus_{i\in I}L_i$  with $I$ nonempty and finite 
and $L_i\cong R((t))$.  For the sake of transparency, we  begin with an abstract discussion that will lead us to the 
Fock representation of a symplectic local system.

\subsection*{Abstract spaces of covacua I}
Let  $A$  be a $R$-subalgebra of $L$ and let $\theta_{A/R}$ have the usual meaning
as the Lie algebra of $R$-derivations $A\to A$. We denote by $A^\perp\subset L$ the annihilator of $A$ 
relative to the residue pairing.
We assume that:
\begin{enumerate}
\item[($A_1$)] as an $R$-algebra, $A$ is flat and of finite type and $A\cap \Ocal =R$,
\item[($A_2$)] the $R$-modules $L/(A+\Ocal)$ and $F: =A^\perp\cap\Ocal $ are 
free of finite rank and the residue pairing induces a perfect pairing $L/(A+\Ocal)\otimes_R F\to R$.
\item[($A_3$)] the universal continuous $R$-derivation  $d: L\to \omega$ maps $A$ to $A^\perp$ and the $A$-dual of the 
resulting $A$-homomorphism  $\Omega_{A/R}\to A^\perp$  is an $R$-isomorphism 
$\Hom_A(A^\perp,A)\cong\theta_{A/R}$.
\end{enumerate}

\begin{remark}\label{rem:}
The example to keep in mind is  the following. Since $R$ is regular local $k$-algebra, it represents a smooth germ 
$(S,o)$. Suppose we are given a family  $\pi: \Ccal\to S$ of smooth projective curves of genus $g$ over this 
germ, endowed with pairwise disjoint sections $(x_i)_{i\in I}$. We let  $\Ocal_i$ be  is the formal completion of 
$\Ocal_{\Ccal}$ along $x_i$, let $L_i$ be obtained from $\Ocal_i$ by inverting a generator for the ideal defining 
$x_i(S)$, and take for  $A$ the $R$-algebra of regular functions on 
$\Ccal^\circ:=\Ccal-\cup_i x_i(S)$ (or rather its isomorphic image in $L=\oplus_i L_i$). It is a classical fact that the 
three properties $A_1,A_2,A_3$ are then satisfied. For instance, $L/(A+\Ocal)$ has according to Weil the interpretation 
of $R^1\pi_*\Ocal_\Ccal$ and hence is free of rank $g$. It is also classical that the annihilator of $A$ in $\omega$ 
is precisely the image of the space relative rational differentials on $\Ccal/S$ that are regular on $\Ccal^\circ$ 
(so in this case  $\Omega_{A/R}\to A^\perp$ is already an isomorphism before dualizing).
\end{remark}

We put $H:= A^\perp/A$. It follows from properties ($A_1$) and  ($A_2$), that the natural map $F\to H$ is an 
embedding with image a Lagrangian subspace.  Recall that $\theta_{A,R}$ denotes  the Lie algebra of 
$k$-derivations $A\to A$  which preserve $R$. The kernel of the natural map $\theta_{A,R}\to \theta_R$ is 
$\theta_{A/R}$ and its image,  is by definition the $R$-submodule  of $k$-derivations  $R\to R$ that extend to one 
of $A$. We denote this image by $\theta_R^A\subset \theta_R$ and  refer to it as the module of 
\emph{liftable derivations}. This module is clearly closed under the Lie bracket. We shall assume that we have 
equality in the generic point,
so that $\theta_R^A$ is as our $\Dfrak$. According to ($A_3$) any element of 
$\theta_{A/R}$ induces the zero map in $H$ and so $\theta_{A,R}$ acts in $H$ (as a  $k$-Lie algebra) through $\theta_{A,R}$.  It is clear that  $\theta_{A,R}\subset \theta_{L,R}$. 

(In the above example, $H$ would represent the first De Rham cohomology module of $\Ccal/S$, $F$ the 
module of relative regular differentials, and we would have  $\theta_R^A=\theta_R$, as every vector field 
germ on $(S,o)$ lifts to rational vector field on $\Ccal$ that is regular on  $\Ccal^\circ$. The Lie action is 
then 
that of covariant derivation of relative cohomology classes. The reason for us to allow $\theta_R^A\not=\theta_R$ 
is because we want to admit the central  fibers of $\Ccal\to S$ to have modest singularities; in that case 
$\theta_R^A$ is the $\theta_R$-stabilizer of  a principal ideal in $R$, the 
\emph{discriminant} ideal of $\pi$.) 

We write $\hat\theta_{A,R}$ for
the preimage of  $\theta_{A,R}$ in $\hat\theta_{L,R}$ and by
$\hat\theta_R^A$ the quotient $\hat\theta_{A,R}/\theta_{A/R}$. 
These are  extensions of $\theta_{A,R}$ resp.\ $\theta_R^A$ 
by $c_0R$. They can be split, but not canonically so. 

Since $A d(A)\subset A^\perp$,
the residue pairing vanishes on $A\times A d(A)$ and hence  $A$ is contained in $\lfrakhat$ as an abelian Lie subalgebra. 
Let $\FF_A:= \FF /A\FF$ denote the space of $A$-covariants. 

\begin{theorem}\label{thm:A}
The following properties hold:
\begin{enumerate}
\item[(i)] The space of covariants $\FF_A$ is naturally identified with the Fock representation $\FF(H,F)$,
\item[(ii)] for  every  $D\in\theta_{A/R}$ there exists a lift $\hat D\in\hat\theta_{A/R}$ such that 
$T(\hat D)$ lies in the closure of 
$A \circ\lfrakhat$ in $\Ubar\lfrakhat$,
\item[(iii)] the  representation of the Lie algebra $\hat\theta_{A,R}$ on 
$\FF$ preserves the submodule $A \FF$ and $\hat\theta_{A,R}$ acts in  $\FF_A$ through 
$\hat\theta_R^A$ by differential operators of degree $\le 1$ (with $c_0$ acting as the identity), 
\item[(iv)] if $\Theta$ is a square root of $\det_R(F)$, then the image of  this action on  $\FF_A$ is  equal to 
the image of the Lie algebra of first order differential operators $\theta_R^A(\Theta)$ (as described in Remark \ref{rem:twistedfock}).
\end{enumerate}
\end{theorem}
\begin{proof}
The proof of the first assertion is straightforward and left to the reader.

Since $L/(A+\Ocal)$ is finitely generated as a $R$-module, we can choose a finite subset
$M\subset L$ such that $L=A +\sum_{f\in M} Rf +\Ocal$.

Now let $D\in \theta_{A/R}$. According to ($A_3$), we may view $D$ as 
a $L$-linear map $\omega\to L$ which maps $A^\perp$ to $A$. This implies that
$\hat C (\hat D)$ lies in the closure of the image of 
$A\otimes_R \lfrakhat+ \lfrakhat\otimes_R A$ 
in $\Ubar \lfrakhat$. It follows that  $\hat C (\hat D)$ has the form
$\hbar r+\sum_{n\ge 1} f_n\circ g_n$ with $r\in R$, one of $f_n,  g_n\in L$  
being in $A$ and the order of $f_n$ smaller than that of
$g_n$ for almost all $n$. In view of the fact that the nonzero elements of $A$ 
are of lower order than those of $\Ocal$ and $f_n\circ g_n\equiv
g_n\circ f_n \pmod{\hbar R}$, we can assume
that all $f_n$ lie in $A$ and so we can arrange that 
$\hat C (\hat D)$ lies in the closure of $A\circ \lfrakhat$. 

For  (iii)  we observe that if $D\in \theta_{A,R}$ and 
$f\in A$, then  $[D,f]=Df$ lies in $A$. 
This shows that $T(\hat D)$ preserves  $A\FF$ and hence acts in $\FF_A$. When $D\in \theta_{A/R}$ and if we choose
$\hat D\in\hat\theta_{A/R}$ as in (ii), then $T(\hat D)$ is clearly zero in $\FF_A$. 
Thus $\hat\theta_{A,R}$ acts in $\FF_A$ through  $\hat\theta_R^A$.

Property (iv) follows from the observation that the action of $\hat\theta_{A,R}$ on $\FF_A\cong\FF (H,F)$ 
evidently has the properties described in  Proposition-Definition \ref{propdef:connection}.
\end{proof}

\subsection*{Abstract spaces of covacua II}
We continue with the setting of the previous subsection. With $\gfrak$ as before we have defined $\FF_\ell(\gfrak)$.
We first consider   the space of 
$A\gfrak$-covariants in $\FF_\ell(\gfrak)$, 
\[
\FF_\ell(\gfrak)_{A\gfrak}:=\FF_\ell(\gfrak)/A\gfrak \FF_\ell(\gfrak).
\] 

\begin{proposition}\label{prop:B}
For  $\hat D\in \hat\theta_{A/R}$, $T_\gfrak(\hat D)$ lies in the closure of 
$A\gfrak\circ\Lghat$ in $\Ubar\Lghat$. The Sugawara representation of the Lie algebra 
$\hat\theta_{A,R}$ on 
$\FF_\ell(\gfrak) $ preserves the submodule $A\gfrak \FF_\ell(\gfrak)\subset \FF_\ell(\gfrak)$ and
acts in the space of $A\gfrak$-covariants in $\FF_\ell(\gfrak)$, $\FF_\ell(\gfrak)_{A\gfrak}$, 
via $\hat\theta_R^A$; this representation is one by differential operators of degree $\le 1$ 
(with $c_0$ acting as multiplication by  $(c+\check{h})^{-1}c\dim\gfrak$).  
\end{proposition}
\begin{proof} The proof is similar to the arguments used to prove Theorem \ref{thm:A}.
Since $D$ maps $A^\perp$ to $A\subset L$,
$1\otimes D$ maps the submodule $\gfrak\otimes A^\perp$ of 
$\gfrak\otimes\omega$ to the submodule $\gfrak\otimes A=A\gfrak$ of 
$\gfrak\otimes L =L\gfrak$. It is clear that  
$\gfrak\otimes A^\perp$ and $A\gfrak$ are each others 
annihilator relative to the pairing $r_\gfrak$. This implies that
$\hat C (\hat D)$ lies in the closure of the image of 
$A\gfrak\otimes_k \Lg+ \Lg\otimes_k A\gfrak$ 
in $\Ubar \Lghat$. It follows that  $\hat C (\hat D)$ has the form
$cr+\sum_\kappa\sum_{n\ge 1} X_\kappa f_{\kappa,n}\circ X_\kappa g_{\kappa,n}$ with $r\in R$, one of 
$f_{\kappa_n},  g_{\kappa,n}\in L$  
being in $A$ and the order of $f_{\kappa_n}$ smaller than that of
$g_{\kappa,n}$ for almost all $\kappa, n$. Since the elements of $A$ have 
order $\le 0$ and $X_\kappa f_{\kappa,n}\circ X_\kappa g_{\kappa,n}\equiv
X_\kappa g_{\kappa,n}\circ X_\kappa f_{\kappa,n}\pmod{cR}$, we can assume
that all $f_{\kappa, n}$ lie in $ A$ and so the first assertion follows.

If $D\in \theta_{A,R}$, then for 
$X\in \gfrak$ and $f\in A$, we have $[D,Xf]=X(Df)$, which is an element
of $A\gfrak$ (since $Df\in A$). This shows that $T_\gfrak(\hat D)$ 
preserves $A\gfrak\FF_\ell(\gfrak)$. If $D\in \theta_{A/R}$, then it follows from the proven part  that 
$T_\gfrak(\hat D)$ maps 
$\FF_\ell(\gfrak)$ to $A\gfrak\FF_\ell(\gfrak)$ and hence induces the zero
map in $\FF_\ell(\gfrak)_{A\gfrak}$. So $\hat\theta_{A,R}$ acts on $\FF_\ell(\gfrak)_{A\gfrak}$ via $\hat\theta_R^A$.
\end{proof}

For what follows we need to briefly review from \cite{kac} the theory of highest weight representations of a 
loop algebra such as $\Lghat$. According to that  theory, the natural analogues  for 
$\Lghat$ of the finite dimensional  irreducible representations  of the  finite dimensional semi-simple 
Lie algebras are obtained as follows, assuming that $I$ is a singleton.
Fix an integer $\ell\ge 0$ and let $V$ be a finite dimensional irreducible representation of $\gfrak$.  
Make $V$ a $k$-representation of  $F^0\Lg$ by letting $c$ act as multiplication by  $\ell$ and by letting 
$\gfrak\otimes_k\Ocal$ act via its projection onto $\gfrak$. If we induce this up to $\Lghat$ we get  a representation 
$\widetilde{\HH}_\ell (V)$ of $\Lghat$  which clearly is a quotient of $\FF_\ell(\gfrak)$.  Its irreducible quotient is denoted 
by $\HH_\ell(V)$. This is integrable as an  $\Lghat$-module: if $Y\in\gfrak$ is nilpotent and $f\in L$, then $Y f$ acts 
locally nilpotently  in $\HH_\ell(V)$ (which means that the latter is a union of finite dimensional $Yf$-invariant 
subspaces in which $Yf$ acts nilpotently). We can be more precise if we fix  a Cartan subalgebra  
$\hfrak\subset\gfrak$ and a system of positive roots $(\alpha_1,\dots ,\alpha_r)$ in  $\hfrak^*$. Let 
$\theta\in\hfrak^*$  the highest root, $\check{\theta}\in\hfrak$ the corresponding coroot and $X\in\gfrak$ a generator of the root space  $\gfrak_\theta$.

\begin{lemma}\label{lemma:}
If $\lambda\in\hfrak^*$ be the highest weight of $V$, then $\HH_\ell (V)$ is zero unless $\lambda (\check{\theta})\le \ell$.
Assuming this inequality, then $\HH_\ell (V)$ can be obtained as the quotient of $U\Lghat$ by the left ideal generated by  $\gfrak\otimes_k\mfrak$, $c-\ell$ and $(Xf)^{1+\ell-\lambda (\check{\theta})}$, where we can take for $f$ any $\Ocal$-generator of $F^{-1}\lfrak$. In fact, the image of $V$ in $\HH_\ell (V)$ (which generates  $\HH_\ell (V)$ as a 
$\Lghat$-representation) is annihilated by all expressions of the form 
$Xf_N\circ\cdots\circ Xf_1$ with $f_k\in F^{-1}\lfrak$ and $N>\ell-\lambda (\check{\theta})$.
\end{lemma}
\begin{proof}
The first assertion is in the literature in the form of an Exercise (12.12 of \cite{kac}). As to the second statement: 
choose variables $u_1,\dots ,u_N$ and observe that $f_u:=f+\sum_k u_kf_k$ is an $\Ocal$-generator of $F^{-1}\lfrak$ for generic $u$. So $V$ is killed by $(Xf_u)^N$ for generic $u$ and hence for all $u$.   By taking the coefficient of $u_1\cdots u_N$ (and using that the $Xf_k$'s commute with each other), we find that $Xf_N\circ\cdots\circ Xf_1$ annihilates $V$.
\end{proof}

Let us call the $k$-span of an $X$ as above  a \emph{highest root line}.
Since the Cartan subalgebras of $\gfrak$ are all conjugate under the adjoint representation, the same is true for 
the highest root lines.

\begin{definition}\label{def:level}
The \emph{level} of a finite dimensional  representation $V$ of $\gfrak$ is the smallest integer $\ell$ for which 
some (or equivalently, any) highest root line $\nfrak$  has the property that $\nfrak^{\ell +1}\subset U\gfrak$ kills $V$.
We denote it by $\ell (V)$.
\end{definition}

It is clear that in terms of the above root data, the set $P_\ell$ of equivalence 
classes of irreducible representations of level $\le \ell$ can be identified with the set of integral weights in a simplex, 
hence is finite. Notice  that  $P_\ell$  is invariant under dualization and more generally, under all outer 
automorphisms of  $\gfrak$. 

Returning to the general case in which $I$ need not be a singleton, we put
$\HH_\ell (V):=\otimes_{i\in I} \HH_\ell (V_i)$. So this is zero unless every $V_i$ is of level $\le\ell$. Inspired by the physicists terminology,  
the $R$-module  $\HH_\ell(V)_{A\gfrak}$ is called 
the space of \emph{covacua} attached to $A$. The following proposition says 
that it is of finite rank and describes  the WZW-connection.

\begin{proposition}[Finiteness]\label{prop:finiterank}
The space $\HH_\ell (V)$ is finitely generated as a $U\! A\gfrak$-module
(so that $\HH_\ell (V) _{A\gfrak}$ is a finitely generated $R$-module).
The Lie algebra $\hat\theta_R^A$ acts on $\HH_\ell (V) _{A\gfrak}$
via the Sugawara representation with  
$c_0$ acting as multiplication by $\frac{\ell}{\ell+\check{h}}\dim\gfrak$. 
\end{proposition}
\begin{proof}
Choose a generator $t_i$ of $\mfrak_i$. The issue being local on $\spec (R)$, we may assume that after localizing $R$,
there exists a finite set 
$\Phi$ of \emph{negative} powers of these generators mapping to an $R$-basis set of $L/(\Ocal +A)$. 
The nilpotent elements of $\gfrak$ span a nontrivial subspace that is invariant under the adjoint action
and hence span all of $\gfrak$. Let $\Xi\subset\gfrak$ be a $k$-basis of $\gfrak$ consisting of nilpotent elements.
Then for a pair $(X,f)\in\Xi\times\Phi$, $Xf$ acts locally nilpotently in $\HH_\ell(V)$ and so there exists
a positive integer $N$ such that the $N$th power of any such element kills the image of 
$\otimes_{i\in I} V_i$ in $\HH_\ell(V)$.

A PBW type of argument then shows that $\HH_\ell (V)$ is the sum of the subspaces
\[
A\gfrak \circ (X_rf_r)^{\circ n_r}\circ \cdots \circ  (X_1f_1)^{\circ n_1}\otimes(\otimes_{i\in I} V_i)\subset \HH_\ell (V)
\]
with  $(X_i,f_i)\in\Xi\times \Phi$ pairwise distinct for $i=1,\dots ,r$, and  $n_1\ge \cdots \ge n_r\ge 0$. 
Since we get a nonzero element only when  $n_1<N$, we thus obtain a finite
collection of $R$-module generators of $\HH_\ell (V)_{A\gfrak}$.
The remaining statements follow from \ref{prop:B}.
\end{proof}

\begin{remark}
We expect the $R$-module $\HH_\ell (V) _{A\gfrak}$ to be flat as well and 
this to be a consequence of a related property for the $UA\gfrak$-module $\HH_\ell (V)$. 
Such a result, or rather an algebraic proof of it, might simplify the argument 
in \cite{tuy} (see Section \ref{section:doublept}  for our version) which shows that the sheaf of covacua attached 
to a degenerating family of pointed curves is locally free.
\end{remark}

\begin{remark}
It is clear from the definition that a system of $\gfrak$-equivariant isomorphisms 
$(\phi_i:V_i\cong V'_i)_{i\in I}$  of finite dimensional irreducible representations induces an 
isomorphism $\phi_*: \HH_\ell (V)_{A\gfrak}\cong \HH_\ell (V')_{A\gfrak}$.  
By Schur's lemma, each $\phi_i$ is unique up to scalar in $k$ and hence the same is true for $\phi_*$. 
We may rigidify the situation by fixing in each representation $V_i$ and $V'_i$ involved  a highest weight 
orbit for the closed connected subgroup of linear transformations whose Lie algebra is the image of $\gfrak$: 
if we require that every $\phi_i$ respects these orbits, then $\phi_i$ is unique.  

We can also say something if we are given a  $\sigma\in \aut(\gfrak)$. This turns every 
representation $V$ of $\gfrak$ into another one (denoted  ${}^\sigma\!V$) that has  the same underlying 
vector space $V$, by letting $X\in\gfrak$ act as $\sigma (X)$ on $V$. The extension $\hat\sigma$ of $\sigma$ to 
$\Lghat$ does the same with $\HH_\ell (V)$. It follows that we have an identification of $\Lghat$-modules:
\begin{multline*}
Y_rf_r\circ\cdots \circ Y_1f_1 {}^\sigma\!\!\!\otimes(\otimes_{i\in I} v_i)\in\HH_\ell ({}^\sigma\!V) \mapsto\\ 
\sigma(Y_r)f_r\circ\cdots \circ\sigma(Y_1)f_1\otimes(\otimes_{i\in I} v_i)\in {}^{\hat\sigma}\!\HH_\ell (V).
\end{multline*}
Since $\sigma$ preserves $A\gfrak$, this descends to an identification 
$\HH_\ell (V^\sigma) _{A\gfrak}\cong\HH_\ell (V)_{A\gfrak}$ of $R$-modules. It is clear from the definition  above  that
this is also equivariant for the Segal-Sugawara representation and hence is an isomorphism of 
$\hat\theta_R^A$-modules.
\end{remark}

\subsection*{Propagation principle} 
The following proposition is a bare version of what is known as the 
\emph{propagation of vacua}; it essentially shows that trivial representations
may be ignored (as long as some representations remain: if all are trivial, then
we can get rid of all but one of them). If we do not care
about the WZW-connection, then this is even true for nontrivial representations 
(a fact that can be found in Beauville \cite{beauville}) so that  
we then essentially reduce the discussion to the case where $I$ is a singleton.

\begin{proposition}\label{prop:singleton}
Let $J\subsetneq I$ be such that  $A$ maps onto $\oplus_{j\in J} L_j/\Ocal_j$. Denote by  
$B\subset A$  the kernel of the map $A\to \oplus_{j\in J} L_j/\mfrak_j\cong R^J$ (evidently an ideal) so that we have a 
surjective Lie homomorphism  $B\gfrak\to(R\otimes_k\gfrak)^{J}$ via 
which $B\gfrak$ acts on $R\otimes_k(\otimes_{j\in J} V_j)$. Then the map of $B\gfrak$-modules 
$\HH _\ell (V|I-J)\otimes_k(\otimes_{j\in J}V_j)\to \HH _\ell (V)$
induces an isomorphism  on covariants:
\[\begin{CD}
\left(\HH_\ell (V|I-J)\otimes_k (\otimes_{j\in J}V_j)\right)_{B\gfrak}@>\cong>> 
\HH_\ell (V)_{A\gfrak}.
\end{CD}\]
If $\theta^{A,B}_R\subset \theta_R^A$ denotes the module of $k$-derivations $R\to R$ that lift to 
$k$-derivations $A\to A$ which preserve $B$ (or equivalently,  $\oplus_{j\in J} \mfrak_j$),  and 
$\hat\theta^{A,B}_R\subset \hat\theta^{A}_R$ stands for the corresponding  extension, then  
the above isomorphism of covariants is compatible with the action of $\hat\theta^{A,B}_R$ on both sides, 
provided that the representations $V_j$ are trivial for $j\in J$.
\end{proposition}
\begin{proof} For the first assertion
it suffices to do the case when $J$ is a singleton $\{ o\}$. The hypotheses clearly imply  that 
$\HH _\ell (V|I-\{ o\})\otimes V_o\to \HH_\ell (V)_{A\gfrak}$ is onto. 
The kernel is easily shown to be $B\gfrak(\HH _\ell (V|I-\{ o\})\otimes V_o)$. 

The second assertion follows in a straightforward manner from our definitions:
if $\bar D\in  \hat\theta^{A,B}_R$, then lift $\bar D$ to a $k$-derivation  $D: A\to A$ which preserves $B$. This implies that $D$ preserves each $\Ocal_j$, $j\in J$. If we choose
a parameter $t_j$ for $\Ocal_j$ so that $\Ocal_j=R((t_j))$,  then $D$ takes in 
$\Ocal_j$ the form $D^{(j)}_\hor+D^{(j)}_\ver$, with $D^{(j)}_\hor$ the extension of $\bar D$ which kills $t_j$ and
$D^{(j)}_\ver=c^{(j)} \partial /\partial t_j$ plus higher order terms with $c^{(j)}\in R$. 
The Sugawara action
of $D^{(j)}_\ver$ on the subspace $V_j\subset \HH_\ell (V_j)$  is up to a factor in $R$
given by $\sum_\kappa t_j^{-1}X_\kappa\circ X_\kappa$. But if $V_j$  is the trivial
representation, then this is evidently zero. The second assertion now follows.
\end{proof}

\begin{remark}
Our discussion of the genus zero case will  show that  the isomorphism
of covariants generally  fails to be compatible relative to the $\hat\theta^{A,B}_R$-action. 
\end{remark}

\begin{remark}
Proposition \ref{prop:singleton}
 is sometimes used in the opposite direction:
if $\mfrak_o\subset A$ is a principal ideal with the property that for a generator $t\in\mfrak_o$, 
the  $\mfrak_o$-adic completion of $A$ gets identified with $R((t))$, then let $\tilde I$ be the disjoint union of 
$I$ and  $\{ o\}$, $\tilde V$ the extension
of $V$ to  $\tilde I$ which assigns to $o$ the trivial representation and  $\tilde A:=A[t^{-1}]$. 
With $(\tilde I,\{o\})$ taking the role of $(I,J)$, we  then find that 
$\HH _\ell (V)_{A\gfrak}\cong \HH _\ell (\tilde V)_{\tilde A\gfrak}$.
\end{remark}

\section{Bundles of covacua}\label{section:cfb}

\subsection*{Spaces of covacua in families}

We specialize the discussion of Section \ref{section:wzw} to a more concrete geometric situation. 
This leads us to sheafify many of the notions we
introduced earlier and  in such cases we shall modify our notation (or its meaning) accordingly. 
Suppose given a proper and  flat morphism between 
$k$-varieties $\pi :\Ccal\to S$ whose base $S$ is smooth and connected and 
whose fibers are 
reduced connected curves that have complete intersection
singularities only (but we do not assume that $\Ccal$ is smooth over $k$).  
Since the family is flat, the arithmetic genus of the fibers is locally constant,
hence constant, say equal to $g$.
We also suppose given disjoint sections $x_i$ of $\pi$, indexed by the 
finite nonempty set $I$  whose union $\cup_{i\in I} x_i(S)$ lies in the smooth part of $\Ccal$ and meets every irreducible component of a 
fiber. The last condition ensures that if $j:\Ccal^\circ:=\Ccal-\cup_{i\in I} 
x_i(S)\subset\Ccal$ is the inclusion, then $\pi j$ is an affine morphism.  

We denote by $(\Ocal_i,\mfrak_i)$ the formal completion of $\Ocal_\Ccal$ along $x_i(S)$, by $\Lcal_i$  the subsheaf of fractions of $\Ocal_i$ with denominator a local generator of  $\mfrak_i$ and by $\Ocal$, $\mfrak$ and $\Lcal$  
the corresponding direct sums. But we keep on using $\omega$, $\theta$, $\hat\theta$ etc.\ for their 
sheafified counterparts.  So these are now all $\Ocal_S$-modules and the residue pairing is also one of 
$\Ocal_S$-modules: $r: \Lcal\times \omega\to\Ocal_S$. 
We write $\Acal$ for $\pi_*j_*j^*\Ocal_{\Ccal}$ (a sheaf of $\Ocal_S$-algebras that is also equal to 
the direct image of $\Ocal_{\Ccal^\circ}$ on $S$)
and often identify this with its image in $\Lcal$. We denote by $\theta_{\Acal/S}$
the sheaf of $\Ocal_S$-derivations $\Acal\to \Acal$ and by $\omega_{\Acal/S}$
for the sheaf $\pi_*j_*j^*\omega_{\Ccal/S}$ (which is also the 
direct image on $S$ of  the relative dualizing sheaf of $\Ccal^\circ/S$;  
if $\Ccal^\circ$ is smooth, this is simply the sheaf of relative differentials). So  $\omega_{\Acal/S}$
is torsion free and embeds therefore  in $\omega$.

\begin{lemma}
The properties  $A_1$, $A_2$ and $A_3$ hold for the sheaf $\Acal$. Precisely,
\begin{enumerate}
\item[($\Acal_1$)] $\Acal$ is as a sheaf of $\Ocal_S$-algebras flat and of finite type, 
\item[($\Acal_2$)] $\Acal\cap\Ocal=\Ocal_S$ and $R^1\pi_*\Ocal_\Ccal=\Lcal/(\Acal+\Ocal)$ is locally free of rank $g$,
\item[($\Acal_3$)] we have $\theta_{\Acal/S}=\Hom_{\Acal}(\omega_{\Acal/S},\Acal)$ and $\omega_{\Acal/S}$ is the annihilator of $\Acal$ with respect to the residue pairing.
\end{enumerate}
\end{lemma} 
\begin{proof} Property $\Acal_1$ is clear. It is also clear that
$\Ocal_S=\pi_*\Ocal_{\Ccal}\to \Acal\cap \Ocal$ is an isomorphism. 
The long exact sequence defined by the functor $\pi_*$ applied to 
the short exact sequence
\[
0\to \Ocal_\Ccal\to j_*j^*\Ocal_{\Ccal} \to \Lcal/\Ocal\to 0
\]
tells us that $R^1\pi_*\Ocal_\Ccal=\Lcal/(\Acal+\Ocal)$; in particular,  
the latter is locally free of rank $g$. Hence $\Acal_2$ holds as well.

In order to verify $\Acal_3$, we note that
$\pi_*\omega_{\Ccal/S}$ is the $\Ocal_S$-dual of 
$R^1\pi_*\Ocal_S$, and hence is locally free of rank $g$.
The first part of $\Acal_3$ follows from the corresponding local property
$\theta_{\Ccal/S}=\Hom_{\Ocal_\Ccal}(\omega_{\Ccal/S},\Ocal_\Ccal)$  by applying
$\pi_*j^*$ to either side. This local property is known to hold for families
of curves with complete intersection singularities. (A proof 
under the assumption that $\Ccal$ is smooth---which does not affect the
generality, since $\pi$ is locally the restriction of that case and both sides
are compatible with base change---runs as follows:
if $j':\Ccal'\subset \Ccal$ denotes the locus where $\pi$ is smooth, then
its complement is of codimension $\ge 2$ everywhere.
Clearly, $\theta_{\Ccal/S}$ is the $\Ocal_{\Ccal}$-dual of 
$\omega_{\Ccal/S}$ on $\Ccal'$ and since both are inert under $j'_*j'{}^*$, 
they are equal everywhere.) 

The last assertion  essentially restates the well-known fact that
the polar part of a rational section of $\omega_{\Ccal/S}$
must have zero residue sum, but can otherwise be arbitrary. More precisely, 
the image of $\omega_{\Acal/S}$ in $\omega/F^1\omega$ is the 
kernel of the residue map $\omega/F^1\omega\to \Ocal_S$. 
The intersection $\omega_{\Acal/S}\cap F^1\omega$ is $\pi_*\omega_{\Ccal/S}$ 
and is hence  locally free of rank $g$.
Since $(F^1\omega)^\perp=\Ocal$, it follows that 
$(\omega_{\Acal/S})^\perp\cap \Ocal$ and $\Lcal/((\omega_{\Acal/S})^\perp+\Ocal)$ 
are locally free of rank 1 and  $g$ respectively. Since $\Acal$ has these properties also 
and is contained in $(\omega_{\Acal/S})^\perp$, we must have 
$\Acal=(\omega_{\Acal/S})^\perp$.
\end{proof}

For what follows one usually supposes 
that the fibers are stable $I$-pointed curves (meaning that  every fiber of $\pi j$
has  only ordinary double points as singularities and has finite 
automorphism group) and is versal (so that the discriminant $\Delta_\pi$ of 
$\pi$ is a reduced normal crossing divisor), but we shall not make these 
assumptions yet. Instead, we assume the considerable weaker property that 
the sections of the sheaf $\theta_S(\log \Delta_\pi)$ of vector fields on $S$ 
tangent to $\Delta_\pi$ lift locally on $S$ to vector fields on $\Ccal$. 
(This is for instance the case  if $\Ccal$ is smooth and $\pi$ is multi-transversal with respect 
to the (Thom) stratification of $\Hom (T\Ccal,\pi^*TS)$ by rank
\cite{looij:book}.) Notice that we have a restriction homomorphism
$\theta_S(\log \Delta_\pi)\otimes \Ocal_{\Delta_\pi}\to \theta_{\Delta_\pi}$.

Let $\theta_{\Ccal, S}\subset\theta_\Ccal$ 
denote the sheaf of derivations which preserve $\pi^*\Ocal_S$.
If we apply $\pi_*j_*j^*$ to the exact sequence 
$0\to\theta_{\Ccal/S}\to\theta_{\Ccal, S}\to 
\theta_{\Ccal, S}/\theta_{\Ccal/S}\to 0$ 
and use our liftability assumption and the fact that $\pi j$ is affine, we get the exact sequence
\[
0\to \theta_{\Acal}\to\theta_{\Acal, S}\to \theta_S(\log \Delta_\pi)\to 0.
\] 
We defined $\hat\theta_{\Acal,S}$ as the preimage of  
$\theta_{\Acal,S}$ in $\hat\theta_{\Lcal,S}$ and 
$\hat\theta_S(\log \Delta_\pi)$ as the quotient 
$\hat\theta_{\Lcal ,S}/\theta_\Acal$. 
These extend $\theta_{\Acal,S}$  and  
$\theta_S$ by $c_0 \Ocal_S$. If we denote the \emph{Hodge bundle}
\[
\lambda:=\lambda(\Ccal/S):= \det (\pi_*\omega_{\Ccal/S}),
\]
then we see that $\hat\theta_S(\log \Delta_\pi)$
may be identified with  the Lie sheaf $\Dcal_1(\lambda)(\log \Delta_\pi)$ of first order differential operators
$\lambda\to\lambda$ which preserve the subsheaf of sections vanishing on $\Delta_\pi$.

Observe that $\Lcal\gfrak=\gfrak\otimes_k\Lcal$ is now a sheaf of 
Lie algebras over  $\Ocal_S$. The same applies to $\hat\lfrak$ and so we have a Virasoro 
extension $\hat\theta_S$ of $\theta_S$ by $c_0\Ocal_S$. We have also defined
${\Acal\gfrak}= \gfrak\otimes_k \Acal$, which is  a Lie subsheaf of 
$\Lcal\gfrak$ as well as of $\Lcalhatg$ and the Fock type $\Lcalhatg$-module
$\Fcal_\ell(\gfrak)$. The will also consider the 
sheaf of ${\Acal\gfrak}$-covariants in the latter, 
\[
\Fcal_\ell(\gfrak)_{\Ccal/S} :=\Fcal_\ell(\gfrak)_{\Acal\gfrak}= {\Acal\gfrak}\Fcal_\ell(\gfrak)\bs\Fcal_\ell(\gfrak).
\]
From Proposition \ref{prop:B} we get:

\begin{corollary}\label{cor:thetaaction}
The representation of the Lie algebra $\hat\theta_{\Acal,S}$ on 
$\Fcal_\ell(\gfrak)$ preserves  ${\Acal\gfrak}\Fcal_\ell(\gfrak) $
and acts on $\Fcal_\ell(\gfrak) _{\Ccal/S}$  via $\hat\theta(\log \Delta_\pi)$  with $c_0$ acting as multiplication by 
$(\ell+\check{h})^{-1}\ell\dim\gfrak$. This construction has a base change property
along any smooth part $S'$ of the discriminant in the  sense
that the residual action of $\hat\theta(\log \Delta_\pi)$ on
$\Fcal_\ell(\gfrak) _{\Ccal_{S'}/S'}\cong \Fcal_\ell(\gfrak) _{\Ccal/S} \otimes\Ocal_{S'}$ factors through 
$\hat\theta_{S'}$.
\end{corollary}

The bundle of integrable representations $\Hcal_\ell (V)$ over $S$ is 
defined in the expected manner: it is obtained as a quotient of 
$\Fcal_\ell(\gfrak)$ in the way $\HH_\ell (V)$ is obtained from $\FF_\ell(\Lghat)$.
We write $\Hcal_\ell(V)_{\Ccal/S} $ for $\Hcal_\ell(V)_{\Acal\gfrak}$. 
The following theorem, which is  mostly a summary of what  we have done so far, is one of the main results of the theory. 

\begin{theorem}[WZW-connection]\label{thm:wzwconn}
The $\Ocal_S$-module $\Hcal_\ell(V)_{\Ccal/S} $ is of finite rank;  
it is also locally free over $S-\Delta_\pi$ and the Lie action of 
$\Dcal_1(\lambda)(\log \Delta_\pi)$  on  $\Hcal_\ell (V)_{\Ccal/S} $ 
defines a logarithmic $\lambda$-flat connection relative to $\Delta_\pi$ of weight $\frac{\ell}{2(\ell +\check{h})}\dim\gfrak$. 
 The same base  change property holds along the smooth part of the  discriminant as 
in Corollary \ref{cor:thetaaction}.
Furthermore, any $\sigma\in\aut(\gfrak)$ determines an isomorphism of $\Dcal_1(\lambda)(\log \Delta_\pi)$-modules 
$\Hcal_\ell({}^\sigma\!V)_{\Ccal/S} \cong \Hcal_\ell(V)_{\Ccal/S} $. 
\end{theorem}  
\begin{proof}
The first assertion follows from \ref{prop:finiterank}. The action of $\hat\theta$ factors (locally) through 
$\Dcal_1(\sqrt{\lambda})(\log \Delta_\pi)$ for some square root $\sqrt{\lambda}$ of $\lambda$ and has then weight 
$(\ell+\check{h})^{-1}\ell\dim\gfrak$. This amounts to an action of $\Dcal_1(\lambda)(\log \Delta_\pi)$ of half that weight. 
The last assertion follows from Corollary \ref{cor:derive}.
The rest is clear except perhaps the  local freeness of $\Hcal_\ell(V)_{\Ccal/S} $ on $S-\Delta_\pi$. But this follows 
from the local existence of a connection in the $\Ocal_S$-module $\Hcal_\ell(V)_{{\Ccal}}$.      
\end{proof}

So if $\Lambda^\times\to S$ denotes the $\GG_m$-bundle  that is associated to $\lambda$, then
we have a flat connection on the pull-back of $\Hcal_\ell (V)_{\Ccal/S}$ to 
$\Lambda^\times|S-\Delta_\pi$ with fiber monodromy scalar multiplication by a root of unity of order 
$\frac{\ell}{2(\ell +\check{h})}\dim\gfrak$.

\subsection*{Propagation principle continued} In the preceding subsection we made the 
assumption throughout that a union of sections of $\Ccal\to S$ is given
to ensure that its complement is affine over $S$. However, the propagation principle
permits us to abandon that assumption. In fact, this leads us to let $\VV$ stand for any 
map which  assigns to every $S$-valued point $x$ of $\Ccal$ an irreducible $\gfrak$-representation 
$\VV_x$ of level $\le \ell$, subject to the condition  that  its \emph{support}, 
$\supp (\VV)$ (i.e., the union of the $x(S)$ for which $\VV_x$  is generically not the trivial representation),  
is a trivial finite cover over $S$ and contained in the locus where $\pi: \Ccal\to S$ is smooth. 
We then might write $\Hcal_\ell (\VV)$ for $\Hcal_\ell(\VV|_{\supp (\VV)})$, but  
since  $\Ccal-\supp (\VV)$ need not be  affine over $S$, this does not yield the right notion of conformal block. 
We can find however, at least locally over $S$, additional pairwise disjoint sections 
of  $\Ccal\to S$ so that the complement $\Ccal^\circ$ of their support and that of $\VV$ 
is affine over $S$. Then we can form $\Hcal_\ell(\VV|\Ccal-\Ccal^\circ)$ and 
Proposition  \ref{prop:singleton}  shows that the resulting  bundle of covacua
$\Hcal_\ell(\VV|\Ccal-\Ccal^\circ)_{(\pi_*\Ocal_{\Ccal^\circ})\gfrak}$ with the 
projective connection is independent of the choices made.  This suggests that we let 
$\Hcal_\ell(\VV)$ resp.\  $\Hcal_\ell (\VV)_{\Ccal/S} $
stand for the sheaf associated to the presheaf
\[
S\supset U\mapsto \varinjlim_{\tilde S} \Hcal_\ell(\VV|_{\tilde S}) \text{  resp.\  }
\varinjlim_{\tilde S} \Hcal_\ell(\VV|_{\tilde S})_{\Ccal _U/U}, 
\]
where $\tilde S$ runs over the unions of pairwise disjoint sections as above.  The latter, when twisted with the dual of $\det (\Ccal/S)$,  has, being 
a limit of presheaves with flat connections,  a flat connection as well. It is clear that in this set-up 
there is also no need anymore to insist 
that the fibers of $\pi$ be connected.

\subsection*{The genus zero case and the KZ-connection}
We here assume $C$ to be isomorphic to $\PP^1$. Let $x_1,\dots ,x_n\in C$ be distinct
and  contain $\supp (\VV)$.  Choose an affine  coordinate  $z$ on $C$ (which identifies $C$ with $\PP^1$)
whose domain contains the $x_i$'s and write $z_i$ for $z(x_i)$. 
Notice that $t_\infty:=z^{-1}$ may serve as a parameter for the local field at  $z=\infty$. 
So if $\HH_\ell(k)$ denotes the representation of $\widehat{\gfrak((z^{-1}))}$ attached to 
the trivial representation $k$ of $\widehat{\gfrak((z^{-1}))}$, then by the propagation principle \ref{prop:singleton} we have 
$\HH_\ell (\VV)_C=(V_1\otimes\cdots\otimes V_n\otimes \HH_\ell(k))_{\gfrak[z]}$, 
where $\gfrak[z]$ acts on $V_i$ for $i\le n$ via its evaluation at $z_i$. 
According to \cite{kac}, the $\gfrak[z]$-homomorphism 
$U(\gfrak[z]) \to \HH_\ell(k)$ is surjective and its kernel is the 
left ideal generated by $(zX)^{1+\ell}$,  where $X\in\gfrak$ generates a highest root line.
This implies that $\HH_\ell (\VV)_{\PP^1}$ can be identified with 
a quotient of the space of $\gfrak$-covariants 
$(V_1\otimes\cdots\otimes V_n)_{\gfrak}$, namely its biggest quotient 
on which  $(\sum_{i=1}^n z_iX^{(i)})^{1+\ell}$ acts trivially (where $X^{(i)}$ acts on
$V_i$ as $X$ and on the other tensor factors $V_j$, $j\not= i$, as the identity).
Now regard $z_1,\dots ,z_n$ as variables. Our first observation is that
a translation in $\CC$ does not affect $\HH_\ell (\VV)_C$: if $a\in\CC$, then
the actions of $\sum_{i=1}^n (z_i+a)X^{(i)}$ and $\sum_{i=1}^n z_iX^{(i)}$ on
$V_1\otimes\cdots\otimes V_n$ differ the action of $aX\in\gfrak$. So we always arrange that $z_1+\cdots +z_n=0$.
Consider  in $\CC^n$ the hyperplane $S_{n-1}$ defined by $z_1+\cdots +z_n=0$ and  
denote by $S_{n-1}^\circ$ the open subset of pairwise distinct $n$-tuples. Then
the trivial family over $S_{n-1}^\circ$,  $\Ccal:=\PP^1\times S_{n-1}^\circ$, comes with $n+1$ 
`tautological' sections (including the one at infinity) so that we also have defined
$\Ccal^\circ$. This determines a sheaf 
$\Hcal_\ell (\VV)_{\Ccal/S_{n-1}^\circ}$ over $S_{n-1}^\circ$. According to the preceding, we have an exact sequence
\[
(V_1\otimes\cdots\otimes V_{n})_{\gfrak}\otimes_k\Ocal_{S_{n-1}^\circ} \to 
(V_1\otimes\cdots\otimes V_{n})_{\gfrak}\otimes_k\Ocal_S \to \Hcal_\ell (\VV)_{\Ccal/S_{n-1}^\circ} \to 0,
\]
where the first map is given by $(\sum_{i=1}^n z_iX^{(i)})^{1+\ell}$.
We  identify  its WZW connection, or rather, a natural lift of that connection to
$V_1\otimes\cdots\otimes V_{n}\otimes_k\Ocal_{S_{n-1}^\circ}$. 
In order to compute the covariant derivative with respect to the vector field 
$\p_i:=\frac{\p}{\p z_i}$ on ${S_{n-1}^\circ}$, we follow our recipe and lift it
to $C\times {S_{n-1}^\circ}$ in the obvious way (with zero component along $C$). 
We continue to denote that lift by $\p_i$ and determine its (Sugawara) action on  $\Hcal_\ell (\VV)$.
We first observe that $\p_i$ is tangent to all the sections, except the $i$th. 
Near that section we decompose it as $(\frac{\p}{\p z}+\p_i)-\frac{\p}{\p z}$, where the first term
is tangent to the $i$th section and the second term is vertical. 
The action of the former is easily understood: its lift to $V_1\otimes\cdots\otimes V_{n}\otimes_k\Ocal_{S_{n-1}^\circ}$ 
acts as derivation with respect to $z_i$. The vertical term, $-\frac{\p}{\p z}$, acts via
the Sugawara representation, that is, it acts on the $i$th slot as 
$-\frac{1}{\ell +\check{h}}\sum_\kappa X_\kappa(z-z_i)^{-1}\circ X_\kappa$
and  as the identity on the others, in other words, acts as 
$-\frac{1}{\ell +\check{h}}\sum_\kappa X^{(i)}_\kappa(z-z_i)^{-1}\circ X^{(i)}_\kappa$.
This action does not induce one in 
$V_1\otimes\cdots\otimes V_{n}\otimes_k\Ocal_{S_{n-1}^\circ}$. 
To make it so, we add to this the action by an element of $\gfrak[\Ccal^\circ]U\Lcalhatg$ 
(which of course will act trivially in $\Hcal_\ell (\VV)_{\Ccal/S_{n-1}^\circ} $), namely 
\[
\frac{1}{\ell +\check{h}}\sum_\kappa X_\kappa (z-z_i)^{-1}\circ X^{(i)}_\kappa 
=\frac{1}{\ell +\check{h}}\sum_{j,\kappa} \frac{1}{z-z_i}X_\kappa^{(j)}\circ X_\kappa^{(i)}.
\]
Doing this for every $i$, then the modification acts in $V_1\otimes\cdots\otimes V_{n}\otimes_k\Ocal_{S_{n-1}^\circ}$ as 
\[
\frac{1}{\ell+\check{h}}\sum_{j\not=i} \frac{1}{z_j-z_i}X_\kappa^{(j)}X_\kappa^{(i)}.
\]
Let us regard the Casimir element $c$ as an element  of $\gfrak\otimes_k\gfrak$, 
and denote by $c^{(i,j)}$ its action in $V_1\otimes\cdots\otimes V_{n}$ on the $i$th and 
$j$th factor (since $c$ is symmetric, we have $c^{(i,j)}=c^{(j,i)}$, so  that we need not worry about the order here).
We conclude that the WZW-connection is induced by the connection
on $V_1\otimes\cdots\otimes V_{n}\otimes_k\Ocal_{S_{n-1}^\circ}$ whose connection form is
\[
\frac{1}{\ell+\check{h}}\sum_{i=1}^n \sum_{j\not=i} \frac{dz_i}{z_j-z_i}c^{(i,j)}=
-\frac{1}{\ell+\check{h}}\sum_{1\le i<j\le n} \frac{d(z_i-z_j)}{z_i-z_j}c^{(i,j)}.
\]
It commutes with the Lie action of  $\gfrak$ on $V_1\otimes\cdots\otimes V_{n}$ and so the connection passes to one on 
$(V_1\otimes\cdots\otimes V_{n})_\gfrak \otimes_k\Ocal_{S_{n-1}^\circ}$.
This lift of the WZW-connection  is known as the \emph{Knizhnik-Zamolodchikov connection}. 
It is not difficult to verify that it is flat (see for instance \cite{kohno}), so that we have not just a projectively flat 
connection, but a genuine one.

\begin{proposition}\label{prop:buildingblocks}
The map 
$(V_1\otimes\cdots\otimes V_{n})_\gfrak\otimes_k\Ocal_{S_{n-1}^\circ}\to \Hcal_\ell(\VV)_{\Ccal/S_{n-1}^\circ}$ 
is an isomorphism for $n=1,2$. Hence for $n=1$ (resp.\ $n=2$),  $\Hcal_\ell(\VV)_{\Ccal/S_{n-1}^\circ}$ is zero unless 
$V_0$ is the trivial representation (resp.\ $V_0$ and $V_1$ are each others dual), in which case 
it can be identified with $\Ocal_{S_{n-1}^\circ}$.
\end{proposition}
\begin{proof}
For $n=1$ this is clear. For $n=2$, the stalk of $\Hcal_\ell(\VV)_{\Ccal/S_{1}^\circ}$  at $(z,-z)$, 
$z\not=0$, can be identified with the image in $(V_1\otimes V_2)_\gfrak$ of the kernel of 
$(zX^{(1)}-zX^{(2)})^{1+\ell}$ acting in 
$V_1\otimes V_2$. Since $X^{(1)}+X^{(2)}$ is zero in $(V_1\otimes V_2)_\gfrak$
and $(X^{(1)})^{1+\ell}$ is zero in $V_1$, this  $(V_1\otimes V_2)_\gfrak$.
\end{proof}

\begin{remark}\label{rem:3holes}
A $3$-pointed genus zero curve $(C\cong\PP^1; x_1,x_2,x_3)$ has no moduli, and so we expect  in this case an  
identification of $\HH_\ell (\VV)_C$ also. Indeed, as is shown in \cite{beauville}, if $V_1,V_2, V_3$ are the associated irreducible $\gfrak$-representations of level $\le \ell$, then $\HH_\ell (\VV)_C$ is naturally identified with the biggest quotient of $V_1\otimes V_2\otimes V_3$ on which both $\gfrak$ and the endomorphisms 
$(z_1X^{(1)}+z_2 X^{(2)}+z_3 X^{(3)})^{1+\ell}$ act trivially for \emph{all} values of $(z_1,z_2,z_3)$. This last  condition is of course equivalent to requiring that $X^p\otimes X^q\otimes X^r$ induces the zero map  whenever $p+q+r>\ell$.
\end{remark}

\section{Factorization}\label{section:doublept}

In this section we consider the case when we are given a family $\pi_o :\Ccal_o\to S_o$ of pointed curves of genus $g$  
with a smooth base germ $S_o=\spec(R_o)$ (so $R_o$ is a regular local ring) and for which 
we are given a section $x_0$ along which $\pi_o$ has an ordinary double point. We
assume that the fibers have no other singularities, in other words, that $\pi_o$ is smooth outside $x_0$.
After possibly making an \'etale base change of degree two we find a partial  normalization 
$\nu: \tilde\Ccal_o\to\Ccal_o$ which separates the branches in the (strong) sense
that $\nu$ is an isomorphism over the complement of $x_0(S_o)$ 
and $x_0$ has two disjoint lifts  to $\Ccal_o$ (which we shall denote by $x_+$ and $x_-$). 
In what follows we simply assume this to be already the case. There are two basic cases:
the \emph{nonseparating case}, where $\tilde\Ccal_o/S_o$ is connected---in that case the fibers have  genus 
$g-1$---and the \emph{separating case}, where $x_+$ and $x_-$ take values in different components 
$\tilde\Ccal_\pm$ of $\tilde C_o$ such that the fiber genera $g_\pm$ of $\tilde C_\pm/S_o$ add up to $g$.
Since the natural base of the WZW-connection is the $\GG_m$-bundle defined by a determinant bundle
(or a fractional power thereof), let us first recall what we get in the present case. The bundle of which 
we take the determinant is the direct image of the relative dualizing sheaf 
$\pi_{o*}\omega_{\Ccal_o/S_o}$. This bundle contains the direct image of $\omega_{\tilde\Ccal_o/S_o}$ 
and the two differ only at $x_o$: an element of   
$\omega_{\tilde\Ccal_o/S_o,x_o}$ when pulled back under $\nu$ may have a simple pole at $x_+$ and 
$x_-$ whose residues add up to zero. So we have a natural exact sequence
\[
0\to \nu_*\omega_{\tilde \Ccal_o/S_o}\to \omega_{\Ccal_o/S_o}\to \Ocal_{S_o}\to 0,
\]
where the last map is defined by taking the residue at $x_+$. If we take the direct image under $\pi_o$, 
we see that we have a natural injection $(\pi_o\nu)_*\omega_{\tilde \Ccal_o/S_o}\to \pi_{o*}\omega_{\Ccal_o/S_o}$. 
It is in fact an isomorphism in the separating case, whereas it has a cokernel naturally isomorphic to 
$R_o$ in the nonseparating case. So after taking determinants we get in either case that 
$\lambda(\Ccal_o/S_o)= \lambda(\tilde\Ccal_o/S_o)$, where it
is understood that in the separating case the right hand side equals $\lambda(\tilde\Ccal_+/S_o)\otimes\lambda(\tilde\Ccal_-/S_o)$.

We now also assume given a representation valued map $\VV_o$ on the smooth part of 
$\Ccal_o$ whose support is contained in a finite union of sections $S_o$ so that  we have defined  
$\Hcal_\ell (\VV_o)_{\Ccal_o/S_o}$. 
A coarse version of the \emph{factorization principle} expresses  this $R_o$-module in  terms of 
a space of covacua attached to the normalization $\tilde\Ccal_o/S_o$. The  more refined
form  describes it as a residue of a module of covacua on a smoothing of $\pi_o$ and takes into account 
the flat connection.

Throughout this section $\Sigma_o\subset \Ccal_o$ is a finite union of sections of $\Ccal_o/S_o$ contained in the smooth part of $\Ccal_o$, 
which contains the support of $\VV_o$ and has the additional property that
its complement $\Ccal^\circ_o:= \Ccal_o-\Sigma_o$ is affine over $S_o$ (this can always be arranged 
by adding some `dummy' sections to the support of $\VV_o$).  We often identify $\Sigma_o$  with its preimage
in $\tilde\Ccal_o$.  Notice that $\tilde\Ccal^\circ_o:=\nu^{-1}\Ccal^\circ_o=\tilde\Ccal_o-\Sigma_o$ is also 
affine over $S_o$, being the normalization of an affine $S_o$-scheme. 
We write $A_o$ resp.\  $\tilde A_o$ for their (coordinate) $R_o$-algebras.

\subsection*{Coarse version of the factorization property} 

Recall that $P_\ell$ denotes 
the set of isomorphism classes of irreducible representations of  $\gfrak$
of level $\le \ell$ and is invariant under dualization: if 
$\mu\in P_\ell$, then $\mu^*\in P_\ell$. 
Let  $V_\mu$  be a $\gfrak$-representation in the equivalence class 
$\mu\in P_\ell$ and choose  $\gfrak$-equivariant dualities 
\[
b_\mu:V_\mu\otimes V_{\mu^*}\to k,
\]
where we assume that $b_{\mu^*}$ is the transpose of $b_\mu$.
Its transpose inverse $\check{b}_\mu\in V_\mu\otimes V_{\mu^*}$ then 
spans the line of $\gfrak$-invariants in $V_\mu\otimes V_{\mu^*}$.

\begin{proposition}\label{prop:fact}
Let $\tilde \VV_{\mu,\mu^*}$  be the representation valued map on $\tilde\Ccal_o$ which is constant equal to 
$V_\mu$ resp.\ $V_{\mu^*}$ on $x_+$  resp.\ $x_-$ and is elsewhere equal to $\VV_o$ 
(via the obvious identification defined by $\nu$). Then the contractions 
$b_\mu:V_\mu\otimes V_{\mu^*}\to k$ define an isomorphism
\[
\begin{CD}
\oplus_{\mu\in P_\ell} \HH_\ell (\tilde \VV_{\mu,\mu^*})_{\tilde \Ccal_o/S_o}
@>\cong>>\HH_\ell (\VV_o)_{\Ccal_o/S_o} .
\end{CD}
\]
\end{proposition}

This is almost a formal consequence of:

\begin{lemma}\label{lemma:bilevel}
Let $M$ be a finite dimensional representation of $\gfrak\times\gfrak$
which is  of level $\le \ell$ relative to both factors. If $M^\delta$ denotes 
that same space viewed as $\gfrak$-module with respect to the 
diagonal embedding $\delta: \gfrak\to\gfrak\times\gfrak$, then the 
contraction $\oplus _{\mu\in P_\ell} M\otimes (V_\mu\boxtimes V_\mu^*)\to M$ 
that on each summand is defined  by $b_\mu$ (the symbol $\boxtimes$ 
stands for the exterior tensor product of representations) 
induces an isomorphism between  covariants: 
\[
\begin{CD}
\oplus _{\mu\in P_\ell} \left(M\otimes (V_\mu\boxtimes 
V_\mu^*)\right)_{\gfrak\times\gfrak}@>\cong>> M^\delta_\gfrak.
\end{CD}
\]
\end{lemma}
\begin{proof}
Without loss of generality we may assume that $M$ is irreducible, or more precisely, 
equal to  $V_\lambda\boxtimes V_{\lambda'}$ for some $\lambda,\lambda'\in P_\ell$.
Then $M^\delta=V_\lambda\otimes V_{\lambda'}$. By Schur's lemma, $M^\delta_\gfrak$ is one-dimensional  if 
$\lambda'=\lambda^*$ and trivial otherwise. That same lemma applied
to $\gfrak\times\gfrak$ shows that 
$(M\otimes (V_\mu\boxtimes V^*_{\mu}))_{\gfrak\times\gfrak}$  
is zero unless $(\lambda,\lambda')=(\mu^*,\mu)$, in which case it is one-dimensional and maps isomorphically to
$M^\delta$.
\end{proof}

\begin{proof}[Proof of \ref{prop:fact}]   Evaluation in $x_0$ resp.\ 
$x_+,x_-$ define epimorphisms $A_o\to R_o$ resp.\ $\tilde A_o\to R_o\oplus R_o$ whose kernels may be 
identified by means  of $\nu$. We denote that common kernel by $\Ical$ and by $B$ the algebra of regular 
functions on the smooth part of $\Ccal^\circ_o$. This is also the algebra of regular functions on the 
complement of the two sections $x_\pm$ 
$\tilde \Ccal^\circ_o$. If $\Ical\gfrak$ has the evident 
meaning, then the argument used to prove Proposition \ref{prop:finiterank} 
shows that $M:=\HH_\ell (\VV_o|\Sigma_o)_{\Ical\gfrak}$ is an $R_o$-module of finite rank. 
It underlies a representation of 
$\gfrak\times \gfrak$ of level $\le \ell$ relative to both factors and is such that 
$M^\delta_\gfrak=\HH_{\ell}(\VV_o)_{A_o\gfrak} =\Hcal _\ell (\VV_o)_{\Ccal_o/S_o} $. The assertion now follows from
Lemma \ref{lemma:bilevel} and the argument used for the propagation principle which shows that  
$(M\otimes (V_\mu\boxtimes 
V_\mu^*))_{R_o\gfrak\times R_o\gfrak}=\HH (\tilde \VV_{\mu, \mu^*})_{B\gfrak}=
\Hcal _\ell (\tilde \VV_{\mu,\mu^*})_{\tilde \Ccal_o/S_o}$.
\end{proof}

\subsection*{A smoothing construction}
In order to motivate the algebraic discussion that will follow, we choose generators 
$t_{\pm}$ of the  ideals of the completed local $R_o$-algebras of $\tilde\Ccal_o$ at  $x_{\pm}$ and explain 
how they determine a \emph{smoothing} of $\Ccal_o/S_o$, that is, a way of making 
$\Ccal_o$ the restriction over $S_o\times \{ o\}$ of a flat morphism $\Ccal\to S$, with 
$S:= S_o\times_k\Delta$  (the spectrum of  $R:=R_o[[\tau]]$) which is smooth over $S-S_o$. 
The construction goes as follows:
 in the product $\tilde \Ccal_o\times \Delta$, blow up 
$x_{\pm}\times\{o\}$ and let $\tilde\Ccal$ be the formal neighborhood of the strict transform of $\tilde \Ccal_o\times\{ o\}$. 
So at the preimage of  $x_{\pm}\times\{o\}$ we have on the strict transform of 
$\tilde \Ccal\times\{ o\}$ the formal  $S_o$-chart $(t_{\pm}, \tau/t_{\pm})$. Now let $\Ccal$ be the quotient of 
$\tilde\Ccal$ obtained by identifying these formal $S_o$-charts up to order: 
$(t_+, \tau/t_+)=(\tau/t_{-},t_-)$, so that $(s_+,s_-):=(t_+,t_-)$ is now a formal $S_o$-chart  of $\Ccal$ on which we have 
$\tau=s_+s_-$ (in either domain  $\tau$ represents the same regular function).  We thus have defined a flat morphism 
$\Ccal\to S_o\times\Delta=S$ (with $\tau$ as second component) with the stated properties.  

\begin{remark}\label{rem:dehntwist}
If we were to work in the complex  analytic category, then we could take for $\Delta$ the complex unit disk.  The fiber of 
$\Ccal/S$ over $(s,\tau)\in S_o\times\Delta$ is then obtained by removing from $C_s$ the union of the two 
disks defined by $|t_{\pm}|\le |\tau|$, followed by identification of the two closed  annuli 
$|\tau|< |t_{\pm}|< 1$ by imposing the identity $t_+t_-=\tau$. 

With a view toward a later application---namely, of extracting a topological quantum 
field theory from the WZW model---we note that there is even a limit if $\tau$ tends to zero if we 
keep its argument fixed. To see this, let us first observe that for $|\tau|<\half$, the fiber is also obtained 
by removal  of the union of the two open disks defined by $|t_{\pm}|< \sqrt{|\tau/2|}$, followed by the above identification of the two closed  annuli  $\sqrt{|\tau/2|}\le  |t_{\pm}|\le \sqrt{|2\tau|}$. Now do a real oriented blow up 
$\hat C_s\to \tilde C_s$ 
of the points $x_\pm (s)\in\tilde C_s$. This means that the polar coordinates associated to $t_\pm$ 
are to be viewed as coordinates for the preimage of its domain on $\hat C_s$: $t_\pm=r_\pm\zeta_\pm$ 
with $|\zeta_\pm|=1$ and $r_\pm\ge 0$  such that the exceptional set $\partial\hat C_s$ is defined by $r_\pm=0$. Notice that $\partial\hat C_s$ is indeed the boundary of a surface; it has two components, each of which comes with a natural principal $U(1)$-action.
If we write $\tau=\varepsilon \zeta$ accordingly with $|\zeta|=1$ 
and $\varepsilon>0$,  then for  $\sqrt{\varepsilon/2}\le r_{\pm} \le \sqrt{2\varepsilon}$, $(r_+,\zeta_+)$ 
must be identified with $(r_-,\zeta_-)$ precisely when $r_+r_-=\varepsilon$ and $\zeta_+\zeta_-=\zeta$. 
This has indeed a continuous extension over $\varepsilon=0$, for then we just identify the two boundary 
circles corresponding to $r_\pm=0$  by insisting that $\zeta_+\zeta_-=\zeta$. We thus obtain a  family 
$\hat\Ccal\to\hat \Delta$ over the real oriented blow up $\hat \Delta\to \Delta$ of $\Delta$ at its origin 
and whose fibers over $\partial\hat\Delta$ are as just described. The  dependence of 
$\hat\Ccal |\partial\hat\Delta$  is a priori on the coordinates $t_\pm$, but it is clear from the construction this 
dependence is in fact only via the (real) ray in $T_{x_+}\hat C_s\otimes T_{x_-}\hat C_s$ defined by 
$\frac{\partial}{\partial t_+}\big |_{x_+}\otimes_\CC\frac{\partial}{\partial t_-}\big |_{x_-}$. The fibers of this family just differ 
by the way we identified the boundary circles and we thus see that the monodromy of the family is a positive Dehn twist 
defined by the welding circle. For later use we note that this construction  takes place in the $C^1$-category: 
$\hat\Ccal$ has a natural $C^1$-structure such that the projection to $\hat\Delta$ is $C^1$.

We should perhaps add that this  has an algebro-geometric incarnation in terms of log structures
and that  $T_{x_+}\tilde C_s\otimes T_{x_-}\tilde C_s$ can be understood as the tangent space of the  
semi-universal deformation of the singular germ $(C_s, x(s))$ (equivalently, our data define a smooth point of the boundary divisor 
of some moduli stack $\overline{\Mod}_{g,n}$ and $T_{x_+}\tilde C_s\otimes T_{x_-}\tilde C_s$ can be identified 
with its normal space). 
\end{remark}

We will denote by $\Sigma$ the image of $\Sigma_o\times\Delta$ in both $\Ccal$ and $\tilde\Ccal$. In either case
it is a union of sections over $S$. The representation valued map $\VV_o$ on $\Ccal_o$ is 
extended to $\Ccal$ in the obvious way (so that its support is contained in $\Sigma$) and we  
denote this extension by $\VV$. We let $A$ stand for $R$-algebra of regular functions on $\Ccal^o:=\Ccal-\Sigma$. Notice  
that $A_o= A/(\tau A)$ and that $A$ embeds in $\tilde A_o[[\tau]]$.

\subsection*{The glueing tensor}
Suppose that in the regular local algebra $R$ we are  given a subalgebra $R_o$ and an element $\tau$ in the  
maximal ideal of $R$ such that $R= R_o[[\tau]]$.  Let $L_+$ and $L_-$ be $R$-algebras, both isomorphic to $R((t))$.  
The `ideal' in $L_\pm$ corresponding to $ tR[[t]]$ is denoted by $\mfrak_\pm$. Let $L:=L_+\oplus L_-$ the direct sum as 
$R$-algebras. We assume given a closed $R$-subalgebra $\Ocal_0\subset L$ with the property that it can 
be topologically generated as a $R_o$-algebra by two generators $s_+, s_-$ of the following type: there exist  
generators $t_\pm$ of $\mfrak_\pm$ such that $s_+=(t_+,\tau/t_-)$ and $s_-=(\tau/t_+,t_-)$.  So an element of 
$\Ocal_0$ will then have the form
\begin{multline*}
\sum_{m\ge 0,n\ge 0}a_{m,n}s_+^m s_-^n= \sum_{m\ge 0,n\ge 0}a_{m,n}(t_+^{m-n}\tau^n,t_-^{n-m}\tau^m)=\\=
\sum_{k\ge 0} \Big(\sum_{m\ge 0} a_{m,k}t_+^{m-k}, \sum_{n\ge 0} a_{k,n}t_-^{n-k}\Big)\tau^k=\\=
\sum_{n>m\ge 0} a_{n,m}\tau^n s_+^{m-n}+\sum_{m\ge 0} a_{m,m}\tau^m+\sum_{m>n\ge 0} a_{n,m}\tau^ms_-^{n-m},
\end{multline*}
with $a_{n,m}\in R_o$. Clearly, the coefficients $a_{m,n}$ can be arbitrary in $R_o$ and the element in question is 
zero only when all $a_{m,n}$ are. So $\Ocal_0$ is a copy of $R_o[[s_+,s_-]]$. 
The last identity shows that $\Ocal_0$ is contained in the $R$-submodule generated by nonpositive powers of 
$s_+$ and $s_-$. A similar argument yields 
the following lemma and so the proof is left as an exercise.

\begin{lemma}\label{lemma:expansion}
Any  continuous $R_o$-derivation of $\Ocal_0$  which preserves $\tau\in\Ocal_0$
extends uniquely to one of $L$. If we let $D^\pm_{k}$ stand for $t_\pm ^{k+1}\frac{\p}{\p t_\pm}$, 
then it has there the form
\[
(D^+_{0}, 0) +\sum_{k\ge 0} \tau^k\Big(\sum_{m\ge 0} a_{m,k}D^+_{m-k},
\sum_{n\ge 0} a_{k,n}D^-_{n-k}\Big),
\]
with $a_{m,n}\in R_o$.
\end{lemma}

We have defined $L\gfrak$ and its central extension $\widehat{L\gfrak}$. For $\mu\in P_\ell$,
let $\HH^{\pm}_\ell(V_\mu)$ denote the 
representation attached to $V_\mu$ of the central extension $\widehat{L_\pm\gfrak}$
of $L_\pm\gfrak$, so that  the $R$-module
$\HH^+_\ell(V_\mu)\otimes_R \HH^-_\ell(V_{\mu^*})$
is one of $\widehat{L\gfrak}$. These representations are defined over
$R_o$ (over $k$ even) and so arise from a base change: 
$\HH^\pm_\ell(V_\mu)=R\otimes_{R_o}\HH^\pm_{o,\ell}(V_\mu)$ and likewise for their tensor product. 
The Casimir element $c$ acts in 
$V_\mu$ as a scalar, a scalar we shall denote by 
$c_\mu$. Observe that $c_{\mu^*}=c_\mu$. Its value is best expressed (and computed) in terms of a
Cartan subalgebra $\hfrak\subset \gfrak$ and a system of positive roots relative to $\hfrak$: if we  
identify $\mu$ with its highest weight in $\hfrak^*$, then 
\[
c_\mu=c(\mu , \mu+2\rho),
\]
where $\rho$ has the customary meaning as the half the sum  of the positive roots. In particular,  $c_\mu$ is a positive rational number (the denominator is in fact at most 3). 

\begin{lemma}\label{lemma:invariant}
There exists a series $\varepsilon^\mu=\sum_{d=0}^\infty \varepsilon^\mu_d\tau^d\in \HH^+_\ell(V_\mu)\otimes_{R_o} 
\HH^-_\ell(V_{\mu^*})[[\tau]]$ (the \emph{glueing tensor}) with constant term 
$\varepsilon^\mu_0=\check{b}_\mu$ that
is annihilated by the image of $\Ocal_0\gfrak$ in $\widehat{L\gfrak}$.  Moreover, any continuous $R$-derivation 
$D$ of $\Ocal_0$ which preserves $\tau$ determines a $\hat D\in\hat \theta$ (relative to the Fock 
construction on the $R$-algebra $L$) with the property that  $\varepsilon^\mu$ is an eigenvector
of $T_\gfrak(\hat D)$ with eigenvalue $-\frac{c_\mu}{2(\ell+\check{h})}$.
\end{lemma}
\begin{proof}
We first observe the generators $t_\pm$ of $\mfrak_\pm$ define 
a grading on all the relevant objects on which we have defined the associated
filtration $F$
(e.g., the degree zero summand of   $\HH_\ell (V_\mu)$ is $R\otimes_kV_\mu$). It is known 
(\cite{kac}, \S\ 9.4) that the  pairing $b_\mu: V_\mu\times V_{\mu^*}\to k$ extends (in fact, in 
a unique manner) to a perfect $R$-pairing
\[
b_\mu : \HH^+_{\ell}(V_\mu)\times
\HH^-_{\ell}(V_{\mu^*})\to R
\]
with the property that 
$b_\mu(Xt_+^n u,u')+b_\mu(u,Xt_-^{-n} u')=0$ for all $X\in\gfrak$ 
and $n\in\ZZ$. This formula implies that the restriction of
$b_\mu$ to  
$\HH^+_\ell(V_\mu)_{-d}\times \HH^-_\ell(V_{\mu^*})_{-d'}$ is zero 
when $d\not= d'$ and is perfect when $d=d'$.  So if $\varepsilon^\mu_d\in
\HH^+_\ell(V_\mu)_{-d}\otimes \HH^-_\ell(V_{\mu^*})_{-d}$ denotes the latter's transpose inverse, 
then we have for all $n\in\ZZ$, $X\in\gfrak$ the following identity in
$\HH^+_\ell(V_\mu)_d\times \HH^-_\ell(V_{\mu^*})_{-d-n}$:
\[
(Xt_+^n\otimes 1)\varepsilon^\mu_{d+n} +(1\otimes Xt_-^{-n})\varepsilon^\mu_{d}=0.
\]
This just says that $(Xt_+^n\otimes 1)+\tau^n(1\otimes Xt_-^{-n})$ kills 
$\varepsilon^\mu:=\sum_{d\ge 0} \varepsilon^\mu_d\tau^d$. Since  $s_+^n=(t_+^n, \tau^nt_-^{-n})$,  this 
amounts to  saying that  $Xs_+^n\in \Ocal_0\gfrak\subset\widehat{L\gfrak}$ kills $\varepsilon^\mu$. 
Likewise for $Xs_-^n$. Since  any element of $\Ocal_0$ lies in the $R$-submodule generated by the 
nonpositive powers of $s_+$ and $s_-$, it follows that $\varepsilon^\mu$ is killed by all of 
$\Ocal_0\gfrak$.

The second statement is proved by a direct computation. If we use Lemma \ref{lemma:expansion} to 
write  $D$ as an operator in $L$, then we find that it suffices to prove: 
 \begin{enumerate}
\item[(i)] $\tau^nT_\gfrak(\hat D^+_{m-n})
-\tau^mT_\gfrak(\hat D^-_{n-m})$ kills $\varepsilon^\mu$ for all $m,n\ge 0$, and
\item[(ii)] $T_\gfrak (\hat D^+_{0})(\varepsilon^\mu)=
-\frac{c_\mu}{2(\ell+\check{h})}\varepsilon^\mu$.
 \end{enumerate} 
As to (i), if we substitute 
\[
T_\gfrak(\hat D^+_{m-n})=-\frac{1}{2(\ell+\check{h})}
\sum_{j\in\ZZ}\sum_\kappa :X_\kappa t_+^{m-n-j}\circ X_\kappa t_+^j:
\]
and do likewise
for $T_\gfrak(\hat D^-_{n-m})$, then this assertion follows easily. 

For (ii) we first observe that $T_\gfrak (\hat D^+_{0})$ 
preserves the grading of
$\HH^+_\ell(V_\mu)$ and acts on $\HH^+_\ell(V_\mu)_0=R\otimes_kV_{\mu}$ as 
$-(2\ell+2\check{h})^{-1}\sum_\kappa X_\kappa\circ X_\kappa$. This is just 
multiplication by $-\frac{c_\mu}{2(\ell+\check{h})}$. 
For an element $u\in \HH^+_\ell(V_\mu)_{-d}$ of the form
$u=Y_rt_+^{-k_r}\circ\cdots\circ Y_1t_+^{-k_1}\circ v$ with $v\in V_\mu$, and 
$Y_\rho\in \gfrak$ (so that $d=k_r+\cdots +k_1$), we have 
\[
T_\gfrak (\hat D^+_{0})(u)=-du+Y_rt_+^{-k_r}\circ\cdots\circ Y_1t_+^{-k_1}\circ
T_\gfrak (\hat D^+_{0})(v)=(-d-\frac{c_\mu}{2(\ell+\check{h})})u.
\]
Since $D^+_{0}(\tau^d)=d\tau^d$, it follows that $\varepsilon^\mu_d\tau^d$ is an eigenvector
of  $T_\gfrak (\hat D^+_{0})$ with eigenvalue $-\frac{c_\mu}{2(\ell+\check{h})}$.
\end{proof}

\subsection*{Finer version of the factorization property}
It is clear that our smoothing  identifies the $R$-module $\HH_\ell (\VV)$ with $\HH_\ell (\VV_o)[[\tau]]$. 
According to Proposition \ref{prop:finiterank},  
$\Hcal_\ell (\VV)_{\Ccal/S} =\HH_\ell (\VV)_{A\gfrak}$
is a finitely generated $R$-module. Since $A_o=A/\tau A$, the reduction of $\HH_\ell (\VV)_{A\gfrak}$
modulo $\tau$ yields $\HH_\ell (\VV_o)_{A_o\gfrak}=\Hcal_\ell (\VV_o)_{\Ccal_o/S_o}$. 
Proposition \ref{prop:fact} identifies the latter with 
$\oplus_{\mu\in P_\ell} \Hcal_\ell (\tilde \VV_{\mu,\mu^*})_{\tilde\Ccal_o/S_o}$.
It is our goal to extend this identification to one of the space of  covacua
$\Hcal_\ell (\VV)_{\Ccal/S}$ with the pull-back  of 
$\oplus_{\mu\in P_\ell} \Hcal_\ell (\tilde \VV_{\mu,\mu^*})_{\Ccal_o/S_o}$ along 
the projection $\pi_{S_o}:S\to S_o$ and to identify the connection on that pull-back. 
This will imply among other things that 
$\Hcal_\ell (\VV)_{\Ccal/S}$  is a free $R$-module.

\begin{theorem}\label{thm:doublept}
The $R$-homomorphism  defined by  tensoring with the glueing tensor,
\begin{gather*}
E=(E_\mu)_\mu : \HH_\ell (\VV)\to
\oplus_{\mu\in P_l} \HH_\ell (\tilde \VV_{\mu,\mu^*})[[\tau]],\\
U=\sum_{k\ge 0} u_k\tau^k\mapsto \Big(U\hat\otimes_R \varepsilon^\mu =\sum_{k,d\ge 0} 
u_k\otimes \varepsilon_d^\mu\tau^{k+d}\Big)_{\mu},
\end{gather*}
is also a map of $A\gfrak$-representations if we let $A\gfrak$ act on the
right hand side via the inclusion
$A\subset \tilde A_o[[\tau]]$. The resulting  $R$-homomorphism of covariants,
\[
E_{\Ccal/S}\!  :\HH_\ell (\VV)_{A\gfrak} \to
\oplus_{\mu\in P_l} \HH_\ell (\tilde \VV_{\mu,\mu^*})_{\tilde A_o\gfrak}[[\tau]],
\]
is an isomorphism (so that $\HH_\ell (\VV)_{A\gfrak}$  is a free 
$R$-module). It is compatible with covariant differentiation with respect to 
$\theta_{S}(\log S_o)=R[[\tau]]\otimes_{R_o}\theta_{R_o}+R[[\tau]]\tau\frac{d}{d\tau}$ 
relative to the lift to $\hat\theta_{S}(\log S_o)$  of Lemma \ref{lemma:invariant}: 
it commutes with the action on $\theta_{R_o}$, whereas
$\tau\frac{d}{d\tau}$ respects each summand  
$H_\ell (\tilde \VV_{\mu,\mu^*})_{\tilde A\gfrak}[[\tau]]$ 
and acts there as the first order differential operator 
$\tau\frac{d}{d\tau}+\frac{c_\mu}{2(\ell+\check{h})}$.
\end{theorem}
\begin{proof}
The first statement is immediate from Lemma \ref{lemma:invariant}.
So the map on covariants is defined and is  $R$-linear. If we reduce 
$E_{\Ccal/S}$ modulo $\tau$,  we  get the map 
\[
\HH_\ell (\VV_o)_{A_o\gfrak}\to 
\oplus_{\mu\in P_l} \HH_\ell (\tilde \VV_{\mu,\mu^*})_{\tilde A_o\gfrak},
\quad 
u\mapsto \sum_{\mu\in P_\ell} u\otimes \varepsilon_0^\mu,
\]
and observe that this is just the inverse of the isomorphism of Proposition \ref{prop:fact}. 
Since the range of $E_{\Ccal/S}  $ is a free $R$-module, this implies that  $E_{\Ccal/S}$ is an isomorphism.

The commutativity  with the action of $\theta_{R_o}$ is clear.
According to Corollary \ref{cor:derive} covariant derivation with respect to $\tau\frac{d}{d\tau}$ in 
$\HH_\ell (\VV)_{\Ccal/S}  $
is defined by means of a $k$-derivation $D$ of  $A$ which lifts $\tau\frac{d}{d\tau}$:
 if we write $D=\tau\frac{d}{d\tau}+\sum_{n\ge 0}\tau^n D^{(n)}$, where $D^{(n)}$ is a vector field on the 
 smooth part of $\Ccal/S$, then the covariant derivative is induced by $T_\gfrak(\hat D)=\tau\frac{d}{d\tau}+
\sum_{n\ge 0}\tau^n T_\gfrak(D^{(n)})$ acting on $\HH_\ell (\VV_o)[[\tau]]$. From the last clause of 
Lemma \ref{lemma:invariant} we get
that when $U\in \HH_\ell (\VV_o)[[\tau]]$,  
\begin{multline*}
T_\gfrak(D) E_\mu(U)=T_\gfrak(D)(U\varepsilon^\mu)=\\
=T_\gfrak(D)(U) \varepsilon^\mu  -\frac{c_\mu}{2(\ell+\check{h})}U \varepsilon^\mu  
=E_\mu T_\gfrak(D)(U)-\frac{c_\mu}{2(\ell+\check{h})}E_\mu(U).
\end{multline*}
Since $T_\gfrak(D)$ acts on $\HH_\ell (\tilde \VV_{\mu,\mu^*})_{\tilde A_o\gfrak}[[\tau]]$ as  
derivation by $\tau\frac{d}{d\tau}$, the last clause follows.
\end{proof}

\begin{corollary}\label{cor:doublept}
The monodromy of the WZW connection acting on $\Hcal_\ell (\VV)_{\Ccal/S}$
has finite order and acts in the summand 
$\Hcal _\ell (\tilde \VV_{\mu,\mu^*})_{\tilde C/S_o}[[\tau]]$ as 
multiplication by the root of unity $\exp(-\pi\sqrt{-1}\frac{c_\mu}{\ell+\check{h}})$. 
\end{corollary}
\begin{proof}
The multivalued  flat sections of $\Hcal_\ell (\VV)_{\Ccal/\Delta}  $
decompose under $E_{\Ccal/\Delta}  $ as a direct sum labeled by $P_\ell$. The summand
corresponding to $\mu$ is the set of solutions of the differential equation
$\tau\frac{d}{d\tau}U+\frac{c_\mu}{2(\ell+\check{h})}U=0$. These are clearly
of the form $u\tau^{-c_\mu/2(\ell+\check{h})}$  with $u\in 
\HH _\ell (\tilde \VV_{\mu,\mu^*})_{\tilde A_o\gfrak}$. If we let $\tau$ run over
the unit circle, then we see that the monodromy is as asserted. Since
$\frac{c_\mu}{\ell+\check{h}}\in\QQ$, it has finite order.
\end{proof}

\begin{remark}\label{rem:pioneconvention}
We use here the convention that the monodromy of the multivalued function $z^\alpha$ is $\exp(2\pi\alpha\sqrt{-1})$
(rather than  $\exp(-2\pi\alpha\sqrt{-1})$). More pedantically: for us the monodromy is a \emph{covariant} rather than a contra-variant functor from the fundamental groupoid to a linear category.
\end{remark}

\section{The modular functor attached to the WZW model}\label{section:modular}

We show here that the results of Section \ref{section:doublept} lead to topological counterparts that take the form of  (what is called) a modular functor in topological  quantum field theory. 

\subsection*{Defining the functor}

For what follows, the most natural setting would probably be that of quasi-conformal surfaces, but we have chosen to work with the more familiar notion of $C^1$-surfaces. This forced us however to introduce the auxiliary notion of an infinitesimal collar below.

The  main objects will be \emph{compact oriented} 
surfaces endowed with a $C^1$-structure, possibly with boundary, but where we assume that each 
boundary component comes with a principal action of the unit circle $U(1)$ that is compatible with the orientation it receives from the surface. In the rest of this paper, we will simply refer to such an object as a \emph{surface}.

An \emph{infinitesimal collar} of a surface is a  inward pointing (nowhere zero) vector field defined on the boundary only with the property that it is locally trivial in the sense that we can find local $C^1$-diffeomorphism $(r,u)$ of a neighborhood of the boundary onto  $[0,\varepsilon)\times U(1)$  which is compatible with the $U(1)$-action on the boundary and takes  the vector field to 
$\p /\p r\vert_{\{0\}\times U(1)}$. The choice of such a vector field determines a basis for each tangent space (the second tangent vector field 
being the derivative of the $U(1)$-action) and so we may  think of this as a first order extension of the given $U(1)$ action.
Suppose given such an infinitesimally 
collared surface $\Sigma$ and two of its boundary components  $B_+, B_-$. Let us call a \emph{glueing map} for this pair
an \emph{anti-isomorphism} $\phi: B_-\to B_+$, that is,   a $C^1$-diffeomorphism with the property that $\phi (ub)=u^{-1}\phi (b)$ for all $b\in B_-$ and $u\in U(1)$. We call it thus, because if we use it to identify $B_-$ with $B_+$, 
then we get a new (infinitesimally collared) surface $\Sigma_\phi$ without the need of making any further choices: 
the $C^1$-structure must be such that the  normal vector fields become each others antipode. 
Similarly, the topological quotient $\check{\Sigma}$ of $\Sigma$  obtained by contracting each of its 
boundary components also acquires a $C^1$-structure: a function on $\check{\Sigma}$  is 
differentiable precisely when its lift to $\Sigma$ is $C^1$ and is such that its derivative evaluated on the infinitesimal 
collar of a boundary component  is the representation of a linear map in polar coordinates.

\begin{definition}
We call a conformal structure on the interior of the infinitesimally collared surface 
$\Sigma$  \emph{admissible} if it is compatible with the given $C^1$-structure as well as with the infinitesimal 
collaring: for every boundary component either  the conformal structure extends to the boundary or extends across its image in $\check{\Sigma}$ and we demand  that in the first case the infinitesimal collaring be perpendicular to the boundary, and that in the second (cuspidal) case it maps to a $U(1)$-orbit in the tangent space. 
\end{definition}

This somewhat unconventional definition is in part motivated by the following observation. A conformal structure on a 
manifold is just a Riemann metric given up to multiplication by a continuous function. More precisely, it is a section of  the 
bundle of positive quadratic forms modulo positive scalars on the tangent bundle. As the fibers of this bundle have a 
convex structure, so has its space of sections. This also holds in the present case with the given boundary conditions, in 
particular the space of  admissible conformal structures  is contractible. And this is still true if we restrict ourselves to the 
admissible conformal structures that are cuspidal at a prescribed union of boundary components. 
This makes it a tractable notion from the point of view of homotopy.

\begin{definition}
A \emph{$\gfrak$-marking}  of a surface  $\Sigma$ consists of giving a map $V$ that assigns to every boundary component of $\Sigma$ a finite dimensional irreducible represention of $\gfrak$  . We then denote the resulting set of data by 
$(\Sigma,V)$. We say that the $\gfrak$-marking is of level $\le \ell$ if $V$ takes values 
in representations of level $\le \ell$.
\end{definition}

Let $(\Sigma,V)$ be $\gfrak$-marked surface. We first suppose $\Sigma$ endowed with an infinitesimal  collaring. Choose an admissible purely cuspidal conformal structure $C$ with respect to this infinitesimal  collaring. Then $\check{\Sigma}$  acquires a conformal structure and hence (since $\check{\Sigma}$ is oriented) the structure of a compact Riemann surface, or equivalently, a nonsingular complex projective curve. We hope the reader forgives us for denoting that curve by $C$ as well. It comes with an injection $\pi_0(\partial\Sigma)\to C$.   If $V$ takes values in representations of level $\le \ell$, then we have defined the space of 
covacua $\HH_\ell(\VV)_C$; otherwise we set $\HH_\ell(\VV)_C=0$. 
For another  choice  of purely cuspidal admissible conformal structure $C'$, we can find a path of such structures $(C_t)_{0\le t\le 1}$ connecting $C$ with $C'$. The projectively flat  connection can be used to identify the corresponding 
projective spaces, and this identification is independent of the choice of path since they belong to the same homotopy class. 

In order to lift this to the actual vector spaces, we need a `rigging' of $\Sigma$ as follows.
Put $g:=\dim H_1(\check{\Sigma};\RR)$ and denote by $\Lcal (\Sigma)\subset \wedge^g H_1(\check{\Sigma};\RR)$ the 
set of $I\in \wedge^g H_1(\check{\Sigma};\RR)$ for which 
$L(I):=\ker(\wedge I: H_1(\check{\Sigma};\RR)\to \wedge^{g+1}H_1(\check{\Sigma};\RR))$ is a Lagrangian subspace (so that
$I$ is a generator of $\wedge^gL(I)$). Let us first assume that $\Sigma$ is connected. It is known that if $g>0$, then $\Lcal(\Sigma)$ is connected, has infinite cyclic fundamental group with a canonical generator
and is an orbit of the symplectic group $\Sp (H_1(\check{\Sigma};\RR))$. For example, if $g=1$, then $\Lcal(\Sigma)=H_1(\check{\Sigma};\RR)-\{ 0\}\cong \RR^2-\{ 0\}$. (If $g=0$, then $\wedge^g H_1(\check{\Sigma};\RR)=\RR$ and so $\Lcal(\Sigma)$ is canonically identified with $\RR-\{ 0\}$.) An element of $\Lcal (\Sigma)$ may arise if $\check{\Sigma}$ is given as the boundary of a compact oriented $3$-manifold $W$: then the kernel of $H_1(\check{\Sigma};\RR)\to H_1(W;\RR)$ is a Lagrangian sublattice and so an orientation of it yields an element of $\Lcal (\Sigma)$.
 
Let $I\in \Lcal(\Sigma)$. We note that every regular differential on $C$ defines by integration a linear map $L(I)\to \CC$ and the basic theory or Riemann surfaces tells us that we thus obtain a complex-linear isomorphism $H^0(C,\omega_{C})\cong \Hom (L(I),\CC)$. Since this induces an isomorphism between $\det H^0(C,\omega_{C})$ and $\Hom_\RR (\det_\RR L(I),\CC)$, the linear form that takes the value $1$ in $I\in \det_\RR L(I)$ yields a generator $I(C)$ of  $\det H^0(C,\omega_{C})$. 
For similar reasons,  the arc $(C_t)_{0\le t\le 1}$  lifts to a section  $t\in [0,1]\mapsto I(C_t)\in\det H^0(C_t,\omega_{C_t})$ of the  determinant bundle and this in turn yields via Theorem \ref{thm:wzwconn} an identification  of  $\HH_\ell(\VV)_C$ with  
$\HH_\ell(\VV)_{C'}$. As this identification is canonical, we now have attached to the  triple $(\Sigma, V, I)$ and the infinitesimal  collaring of $\Sigma$  a well-defined finite dimensional complex vector space $H_\ell(\Sigma, V,I)$.  Actually, the infinitesimal collaring is irrelevant, for the infinitesimal collarings  make up an affine space over the vector space of vector fields on $\partial\Sigma$ and hence form a contractible set. 

For $g=0$, we shall always take for $I\in \Lcal (\Sigma)$ the canonical element that corresponds to $1$ under the identification
$\Lcal (\Sigma)\cong \RR-\{0\}$ so that we then have a well-defined vector space $H_\ell (\Sigma, V)$.
Proposition \ref{prop:buildingblocks} tells us what we get in some simple cases:

\begin{proposition}\label{prop:diskcylinder}
For $\Sigma$ a disk (resp.\  a cylinder), $H_\ell(\Sigma, V)$ is zero unless $V$ is the trivial representation (resp.\ the two representations attached to the boundary  are each other's contra-gradient), in which case it is canonically equal to $\CC$. 
\end{proposition}

Now drop the assumption that $\Sigma$ be connected and let $\Sigma_1,\dots ,\Sigma_r$ enumerate its distinct connected components. If $I\in \Lcal (\Sigma)$ corresponds to $I_1\otimes_\RR\cdots \otimes_\RR I_r$ with $I_k\in \Lcal (\Sigma_k)$, then the tensor product $H_\ell(\Sigma_1, V_1,I_1)\otimes_\CC\cdots\otimes_\CC H_\ell(\Sigma_r, V_r,I_r)$
only depends on $I_1\otimes_\RR\cdots \otimes_\RR I_r$ and so the preceding generalizes  if we
let $H_\ell(\Sigma, V,I)$  be this tensor product. We thus find:

\begin{theorem}\label{thm:modfunctor}
Let $(\Sigma,V)$ be a $\gfrak$-marked surface of level $\le \ell$. Then we have naturally defined on the Lagrangian manifold $\Lcal (\Sigma)$ a local system $\HH_\ell(\Sigma,V)$ whose stalk at $I\in \Lcal (\Sigma)$ is $H_\ell(\Sigma, V,I)$. This construction is functorial with respect to automorphisms of $\gfrak$ so that for every $\sigma\in\aut(\gfrak)$ we have a natural isomorphism 
$\HH_\ell(\Sigma, {}^\sigma\!V)\cong \HH_\ell(\Sigma,V)$.
\end{theorem}
\begin{proof}
The last assertion follows from the last clause of Theorem \ref{thm:wzwconn}, the rest is clear from the preceding discussion.
\end{proof}

\begin{remark}\label{rem:duality}
The natural involution of $\gfrak$ with respect to a choice of root data takes every finite dimensional $\gfrak$-representation  into one equivalent to its contra-gradient. So for such an involution $\sigma$ we obtain an isomorphism between $H_\ell(\Sigma, V^*,I)$ and $H_\ell(\Sigma, V,I)$, but beware that this  involution is only unique up to inner automorphism. However, one expects that there exists  a canonical perfect pairing (which therefore does not involve a choice of $\sigma$) $H_\ell(\overline{\Sigma}, V^*,I)\otimes H_\ell(\Sigma, V,I)\to \CC$, where $\overline{\Sigma}$ stands for
$\Sigma$ with the opposite orientation.
\end{remark}

\subsection*{Action of the centrally extended  mapping class group}
Let  $\Gamma (\Sigma)$ denote  the  group of orientation preserving isotopies of $\Sigma$ which leave each of its  components and each boundary component invariant (but not necessarily point-wise). This is isomorphic to the usual mapping class group of the pair
consisting of $\check{\Sigma}$ and the finite subset of $\check{\Sigma}$ that appears as the image of $\pi_0(\partial\Sigma)$. 
The above lemma shows that if $(\Sigma, V)$ is a $\gfrak$-marked surface, then for every $I\in\Lcal (\Sigma)$, the mapping class $\phi\in\Gamma (\Sigma)$ gives rise an isomorphism  $\phi_\#: H_\ell(\Sigma, V,I)\to H_\ell(\Sigma, V,\phi_*I)$. In other words, $\phi$ induces an automorphism of the local system $\HH_\ell (\Sigma,V)$.

Assume $\Sigma$ connected and $g>0$ so that $\Lcal (\Sigma)$ has infinite cyclic fundamental group.
Fix a universal cover $\tilde \Lcal (\Sigma)\to \Lcal (\Sigma)$ and denote by $\tilde H_\ell(\Sigma,V)$ the space of sections of the pull-back of $\HH(\Sigma,V)$ to this cover.
The pairs $(\phi, \tilde \phi_\#)$ with $\phi$ a mapping class and $\tilde \phi_\#\in \aut (\tilde \Lcal (\Sigma))$  a lift of $\phi_\#$ define a central  
extension $\tilde\Gamma (\Sigma)\to \Gamma (\Sigma)$ of the mapping class group by 
$\ZZ$. We have arranged things in such a manner that this extension acts on 
$\tilde H_\ell(\Sigma,V)$ with the central element $2(\ell +\check{h})\in\ZZ$ acting trivially. The central extension  is clearly one that already lives 
on the automorphism group  of $H_1(\check{\Sigma})$ (an integral symplectic group of genus $g$). 
The latter is known to  produce the universal central extension of the symplectic group. It has an abstract  description in terms of a $2$-cocyle, known as the Maslov index.

\subsection*{The glueing property}

Let $B_+$ and $B_-$ be distinct boundary components of $\Sigma$ and $\phi: B_+\to B_-$ a glueing map so that we have an associated quotient surface $\Sigma_\phi$. We show that there is a natural embedding
of $\Lcal (\Sigma_\phi)$ in $\Lcal (\Sigma)$. If $B_+$ and $B_-$ lie on distinct components, then
we have a natural identification between the symplectic lattices $H_1(\check\Sigma)$ and $H_1((\Sigma_\phi)\check{})$ and so we also have a natural identification of $\Lcal(\Sigma)$ with $\Lcal(\Sigma_\phi)$. If $B_+$ and $B_-$ lie on the same component, then their common image 
$B$ in $\Sigma_\phi$ has the property that the image $\la B\ra $ of $H_1(B)\to H_1((\Sigma_\phi)\check{})$ is a primitive rank one sublattice. If we denote by $\la B\ra ^\perp$ the annihilator  of $\la B\ra$ with respect to the intersection pairing, then $\la B\ra\subset \la B\ra^\perp$ and the symplectic lattice $\la B\ra ^\perp/\la B\ra $ is naturally identified with $H_1(\check\Sigma)$. Since the intersection pairing identifies $H_1((\Sigma_\phi)\check{})/\la B\ra ^\perp$ with the dual of $\la B\ra $, we see that we have a natural embedding of $\Lcal (\Sigma_\phi)$ in $\Lcal (\Sigma)$.  

If we now combine the discussion in Remark \ref{rem:dehntwist} with Theorem \ref{thm:doublept}, we obtain

\begin{theorem}[Glueing property]\label{thm:sew}
Suppose we have endowed $\Sigma_\phi$ with a $\gfrak$-marking, i.e., a map $V:\pi_0(\partial\Sigma_{\phi})=\pi_0(\partial\Sigma)-\{\{B_+\},\{B_-\}\}\to P_\ell$. 
For $\mu\in P_\ell$, denote by $V_{\mu,\mu^*}: \pi_0(\partial\Sigma)\to P_\ell$ the extension of $V$ to a $\gfrak$-marking of $\Sigma$ which assigns to $B_+$ resp.\ $B_-$ the value $\lambda$ resp.\ $\lambda^*$. Then 
the local system $\HH_\ell (\Sigma_{\phi},V)$ on $\Lcal (\Sigma_\phi)$ can be naturally identified with the restriction of $\oplus_{\lambda\in P_\ell} \HH_\ell (\Sigma,V_{\mu,\mu^*})$ with respect to the
embedding of $\Lcal (\Sigma_\phi)$ in $\Lcal (\Sigma)$ defined above. 
Under this identification, the mapping class of $\Sigma_{\phi}$ obtained by the glueing maps
$\{\zeta\phi\}_{\zeta\in U(1)}$ (a Dehn twist) acts on the summand  $H_\ell (\Sigma,V_{\mu,\mu^*})$ as scalar multiplication by $\exp\big(-\frac{\pi \sqrt{-1}c_\mu}{\ell +\check{h}}\big)$.
\end{theorem}

By repeated application of Theorem \ref{thm:sew} in the opposite direction we can thus completely recover $\HH_\ell (\Sigma, V)$ from a pair of pants decomposition of $\Sigma$: such a decomposition has a 3-holed sphere (a pair of pants) as its basic building block, a case that is taken care of by Remark \ref{rem:3holes}. In particular we obtain a formula, at least in principle, for its dimension, known as the \emph{Verlinde formula}.  This process is nicely formalized by the notion of a fusion ring (see \cite{beauville}). But if we wish to deal with the modular functor itself, then we are led to the representation theory of quantum groups.  As we mentioned in the introduction, this has applications in knot theory via a threedimensional topological quantum field theory. For most of this we refer to the monograph of Turaev 
\cite{turaev}.
\\


\begin{thebibliography}{99}

\bibitem{jk1}
J.E.~Andersen, K.~Ueno:
\emph{Abelian conformal field theory and determinant bundles}, 
Internat.\ J.\ Math.\ 18 (2007),  919--993. 

\bibitem{jk2}
J.E.~Andersen, K.~Ueno:
\emph{Geometric construction of modular functors from conformal field theory},
J.\ Knot Theory Ramifications 16 (2007), 127--202. 


\bibitem{bk}
B.~Bakalov, A.~Kirillov, Jr.:
\emph{Lectures on tensor categories and modular functors.}
University Lecture Series, 21. American Mathematical Society, Providence, RI, 2001. 

\bibitem{beauville}
A.~Beauville:
\emph{Conformal blocks, fusion rules and the Verlinde formula,}
Proceedings of the Hirzebruch 65 Conference on Algebraic Geometry,  
75--96, Israel Math.\ Conf.\ Proc., 9, Bar-Ilan Univ., Ramat Gan (1996).

\bibitem{bms}
A.A.~Beilinson, Yu.I.~Manin, V.V.~Schechtman: 
\emph{Sheaves of the Virasoro and Neveu-Schwarz algebras}, in: 
\emph{$K$-theory, arithmetic and geometry (Moscow, 1984--1986)}, 52--66, 
Lecture Notes in Math., 1289, Springer, Berlin (1987). 

\bibitem{bl}
A.~Boer, E.~Looijenga:
\emph{On the unitary nature of abelian conformal blocks,}
J. Geom. Phys. 60 (2010), no. 2, 205--218.

\bibitem{faltings}
G.~Faltings:
\emph{A proof for the Verlinde formula,}
J.\ Algebraic Geom.\ 3 (1994), no.\ 2, 347--374. 
 
\bibitem{hitchin}
N.J.~Hitchin:
\emph{Flat connections and geometric quantization,}
Comm.\ Math.\ Phys.\ 131 (1990), 347--380. 


\bibitem{kac}
V.G.~Kac: 
\emph{Infinite-dimensional Lie algebras,} 
3rd ed. Cambridge University Press, Cambridge (1990).

\bibitem{kacraina} 
V.G.~Kac, A.K. Raina:
 \emph{Bombay lectures on highest weight representations of 
 infinite-dimensional Lie algebras,} Advanced Series in Mathematical 
 Physics, 2. World Scientific Publishing Co., Inc., Teaneck, NJ (1987).


\bibitem{kohno} 
T.~Kohno: 
\emph{Integrable connections related to Manin and Schechtman's higher 
braid groups,}  Ill.~J.~Math.\ 34 (1990), 476--484.

\bibitem{laszlo} 
Y.~Laszlo: 
\emph{Hitchin's and WZW connections are the same,} 
J.\ Differential Geom. 49 (1998), 547--576. 

\bibitem{ls} 
Y.~Laszlo, C.~Sorger: 
\emph{The line bundles on the moduli of parabolic $G$-bundles over curves and their sections,}  
Ann.\ Sci.\ \'Ecole Norm.\ Sup.\ (4) 30 (1997), 499--525. 

\bibitem{looij:book}
 E.J.N.~Looijenga: 
 \emph{Isolated singular points on complete intersections,} 
 LMS Lecture Note Series, 77. Cambridge University Press, Cambridge (1984). 
 
\bibitem{schsch}
P.~Scheinost, M.~Schottenloher: 
\emph{Metaplectic quantization of the moduli spaces of flat and parabolic bundles,} 
J.\ Reine Angew.\ Math.\ 466 (1995), 145-- 219.

\bibitem{segal}
G.~Segal:
\emph{The definition of conformal field theory,} in: Topology, geometry and quantum field theory, 421--577, London Math. Soc. Lecture Note Ser., 308, Cambridge Univ. Press, Cambridge, 2004.

\bibitem{sorger} 
Ch.~Sorger: 
\emph{La formule de Verlinde,} S\'em.\ Bourbaki Exp.\ 794,
Ast\'erisque  237 (1996), 87--114.

\bibitem{suntsai}
X.~Sun, I-Hsun Tsai:
\emph{Hitchin's connection and differential operators with values in the determinant bundle,}
J. Differential Geom. 67 (2004), 335--376. 

\bibitem{tsuchimoto} 
Y.~Tsuchimoto:
\emph{On the coordinate-free description of the conformal blocks,}  
J.~Math.\ Kyoto Univ.\ 33 (1993),  no. 1, 29--49. 

\bibitem{tuy}
A.~Tsuchiya, K.~Ueno, Y.~Yamada:
\emph{Conformal field theory on universal family of stable curves with 
gauge symmetries,} in:  Integrable systems in quantum field theory and 
statistical mechanics, 459--566, Adv. Stud. Pure Math., 19,
Academic Press, Boston, MA (1989).

\bibitem{turaev}
V.G.~Turaev,
\emph{Quantum invariants of knots and 3-manifolds,}
de Gruyter Studies in Mathematics, 18. Walter de Gruyter \&\ Co., Berlin, 1994.

\bibitem{ueno1}
K.~Ueno: 
\emph{$\mathbf{Q}$-structure of conformal field theory with gauge symmetries,} in:
The moduli space of curves (Texel Island, 1994), 511--531, 
Progr. Math., 129, Birkh\" auser Boston, Boston, MA, 1995. 

\bibitem{ueno2}
K.~Ueno:
\emph{Conformal field theory and modular functor,} in  Advances in algebra and combinatorics, 335--352, World Sci. Publ., Hackensack, NJ, 2008. 

\bibitem{ueno3}
K.~Ueno:
\emph{Conformal field theory with gauge symmetry.} Fields Institute Monographs, 24. American Mathematical Society, Providence, RI; Fields Institute for Research in Mathematical Sciences, Toronto, ON, 2008. 


\end{thebibliography}
\end{document}